\documentclass{amsart}

\usepackage{float}
\usepackage{amsmath}
\usepackage{amsfonts}
\usepackage{amssymb}
\usepackage{graphicx}
\usepackage{hyperref}
\usepackage{mathrsfs}
\usepackage{pb-diagram}
\usepackage{epstopdf}
\usepackage{amsmath,amsfonts,amsthm,enumerate,amscd,latexsym,curves}
\usepackage{bbm}
\usepackage[mathscr]{eucal}
\usepackage{epsfig,epsf,xypic,epic}
\usepackage{caption}
\usepackage{subcaption}
\usepackage{tikz}
\usetikzlibrary{matrix,arrows,decorations.pathmorphing}

\newtheorem{theorem}{Theorem}[section]

\newtheorem{lemma}[theorem]{Lemma}

\newtheorem{definition}[theorem]{Definition}

\newtheorem{example}[theorem]{Example}

\newtheorem{proposition}[theorem]{Proposition}

\newtheorem{remark}[theorem]{Remark}

\newcommand{\dbar}{\bar{\partial}}

\newcommand{\pd}[2]{\frac{\partial #1}{\partial #2}}
\newcommand{\secpd}[3]{\frac{{\partial}^2 #1}{{\partial #2}{\partial #3}}}

\newcommand{\inner}[1]{\langle #1\rangle}

\newcommand{\bb}[1]{\mathbb{#1}}
\newcommand{\cu}[1]{\mathcal{#1}}
\newcommand{\til}[1]{\widetilde{#1}}

\begin{document}

\title[Geometric Quantization via SYZ Transforms]{Geometric Quantization via SYZ Transforms}
\author[K. Chan]{Kwokwai Chan}
\address{Department of Mathematics\\ The Chinese University of Hong Kong\\ Shatin\\ Hong Kong}
\email{kwchan@math.cuhk.edu.hk}
\author[Y.-H. Suen]{Yat-Hin Suen}
\address{Center for Geometry and Physics\\ Institute for Basic Science (IBS)\\ Pohang 37673\\ Republic of Korea}
\email{yhsuen@ibs.re.kr}

\date{\today}

\begin{abstract}
The so-called quantization problem in geometric quantization is asking whether the space of wave functions is independent of the choice of polarization. In this paper, we apply SYZ transforms to solve the quantization problem in two cases:
\begin{enumerate}
\item
semi-flat Lagrangian torus fibrations over complete compact integral affine manifolds, and
\item
projective toric manifolds.
\end{enumerate}
More precisely, we prove that the space of wave functions associated to the real polarization is canonically isomorphic to that associated to a complex polarization via SYZ transforms in both cases.
\end{abstract}

\maketitle

\tableofcontents

\section{Introduction}

\subsection{The quantization problem}

Physicists are keen on quantizing classical theories. Given any classical theory, they aim at finding a quantum counterpart whose classical limit recovers the original classical theory. This process is called quantization. Geometric quantization is a mathematical approach to constructing such quantum theories.

A system of classical mechanics is a phase space formalized as a symplectic manifold, so we let $(\check{X},\check{\omega})$ be a symplectic manifold such that $\check{\omega}$ represents an integral class $[\check{\omega}] \in H^2(\check{X}; \mathbb{Z})$. Then there exists a Hermitian line bundle $\cu{L} \to \check{X}$ which admits a unitary connection $\nabla$ so that
$$\frac{i}{2\pi}F_{\nabla} = \check{\omega}.$$
Such a pair $(\cu{L},\nabla)$ is called a {\em prequantum line bundle} (Definition \ref{def:prequantum}).
Given such a prequantum system, physicists would like to associate a Hilbert space $\cu{H}$ to represent the wave functions or quantum states. A natural choice is to take $\cu{H}$ as the ($L^2$-completion of the) space of all sections of a line bundle $\cu{L}$. However, this space is in general too big to capture the actual physics. To obtain a Hilbert space of reasonable size, physicists introduced the notion of polarization.

A {\em polarization} is an involutive Lagrangian subbundle $P$ of the complexified tangent bundle $T_{\bb{C}}\check{X}$ of $\check{X}$. Given a polarization, one can then define $\cu{H}$ as the space $\Gamma_P(\check{X},\cu{L})$ of {\em polarized sections} -- sections that are covariantly constant along $P$ (Definition \ref{def:pola}).
When $(\check{X}, \check{J}, \check{\omega})$ is a K\"ahler manifold admitting a Lagrangian fibration $\check{p}: \check{X} \to B$, there are two natural choices for the polarization, namely, the {\em real polarization} induced by the Lagrangian fibration $\check{p}$ and the {\em complex polarization} induced by the complex structure $\check{J}$. But physicists believe that the quantum theory should be independent of the polarization we chose, so it is desirable that the space of polarized sections for the real polarization is canonically isomorphic to that for the complex polarization. This is known as the {\em quantization problem}.

In this paper, we show how {\em SYZ transforms}, which are Fourier--type transforms responsible for the interchange between complex-geometric data on $\check{X}$ and symplectic-geometric data on the mirror $X$ \cite{SYZ}, can be applied to solve the quantization problem. We demonstrate this idea in two cases (Section \ref{ch:SYZ_geometric_quantization}):
\begin{enumerate}
\item
semi-flat Lagrangian torus fibrations over complete compact integral affine manifolds (see Definition \ref{def:complete}), and
\item
projective toric manifolds.
\end{enumerate}

\subsection{Applying SYZ}

Let us explain why SYZ transforms can help to solve the quantization problem. We begin with a semi-flat Lagrangian torus fibration $\check{p}:\check{X}\to B$ which admits a Lagrangian section.\footnote{In general there are singular or degenerate fibers in a Lagrangian torus fibration and the semi-flat one is obtained by restricting to the smooth locus.}
The {\em SYZ mirror} of $\check{X}$ is simply given by the total space of the fiberwise dual torus fibration $p: X \to B$. Exploiting the integral affine structure on the base $B$, we have the identifications
\begin{equation*}
\check{X} \cong TB/\Lambda, \quad X \cong T^*B/\Lambda^*,
\end{equation*}
together with the fiberwise dual fibrations
\begin{equation*}
\xymatrix{
\check{X} \ar[dr]_{\check{p}} & & X \ar[dl]^{p}\\
& B &}
\end{equation*}
This is the toy model for SYZ mirror symmetry \cite{SYZ, Leung05}.

Since a torus fiber $F_x := p^{-1}(x)$ is dual to the corresponding fiber $\check{F}_x := \check{p}^{-1}(x)$ (for some point $x \in B$), geometrically $F_x$ is parametrizing the flat $U(1)$-connections over $\check{F}_x$. So a section of $p$ gives a $U(1)$-connection over each fiber of $\check{p}$ and they combine to produce a $U(1)$-connection over the total space $\check{X}$. A simple yet fundamental observation of Arinkin-Polishchuk \cite{AP} and Leung-Yau-Zaslow \cite{LYZ} is that the section $L$ of $p$ is {\em Lagrangian} if and only if the corresponding $U(1)$-connection defines a {\em holomorphic} line bundle $\check{L}$ over $\check{X}$. We call the map $L \mapsto \check{L}$ the {\em SYZ transform}.

Every Hermitian holomorphic line bundle over $\check{X}$ comes from this construction. In particular, we can consider a prequantum line bundle
$$(\cu{L},\nabla)=(\check{L},\nabla_{\check{L}})$$
given by the SYZ transform of some Lagrangian section $L$ of $p: X \to B$.

\'Sniatycki \cite{cohomology_GQ} proved that the space $\Gamma_{P_\bb{R}}(\check{X},\check{L})$ of polarized sections for the real polarization $P_\bb{R}$ (or simply the space of real polarized sections) can be identified with the $\bb{C}$-vector space spanned by the so-called {\em Bohr-Sommerfeld fibers} of $\check{p}: \check{X} \to B$ which are, by definition, fibers of $\check{p}$ over which $\nabla_{\check{L}}$ restricts to the trivial connection (see Definition \ref{def:BS}):
$$\Gamma_{P_{\bb{R}}}(\check{X},\check{L}) \cong \bigoplus_{\check{F}_x:\text{BS fiber}}\bb{C}\cdot \check{F}_x.$$
Under the SYZ transform, these fibers correspond precisely to the intersection points between $L$ and the zero section $L_0$ of the dual fibration $p: X \to B$, so we have
\begin{equation}\label{eqn:SYZ1}
\bigoplus_{\check{F}_x:\text{BS fiber}}\bb{C}\cdot \check{F}_x \cong \bigoplus_{p\in L_0\cap L}\bb{C}\cdot p
\end{equation}
via the SYZ transform.

On the other hand, the space $\Gamma_{P_{\bb{C}}}(\check{X},\check{L})$ of polarized sections for the complex polarization $P_\bb{C}$  (or simply the space of complex polarized sections) is given by the space $H^0(\check{X}, \check{L})$ of holomorphic sections of $\check{L}$. Since $\check{L}$ is ample, we have $\text{Ext}^{\bullet}(\mathcal{O}_{\check{X}}, \check{L}) = H^0(\check{X}, \check{L})$. Then the homological mirror symmetry (HMS) conjecture of Kontsevich \cite{HMS} together with its compatibility with the SYZ conjecture \cite{SYZ} implies the isomorphism
$$H^0(\check{X}, \check{L}) \cong HF^{\bullet}(L_0, L),$$
again via the SYZ transform.
If the intersection of $L_0$ and $L$ is nice enough (e.g. if all the intersection points of $L_0$ and $L$ are of Maslox index 0 so that the Floer differential $\mathfrak{m}_1$ is 0), then we further have
$$HF^{\bullet}(L_0, L) \cong \bigoplus_{p\in L_0\cap L}\bb{C}\cdot p.$$

Combining, we arrive at the canonical isomorphism between the spaces of polarized sections
$$\Gamma_{P_{\bb{R}}}(\check{X},\check{L}) \cong \Gamma_{P_{\bb{C}}}(\check{X},\check{L}),$$
obtained by applying SYZ transforms {\em twice}.

Note that the space $HF^{\bullet}(L_0, L)$ in the above heuristic argument is only playing an auxiliary role. In fact, instead of using $HF^{\bullet}(L_0, L)$ and alluding to the HMS conjecture, we will consider the space
$$ker(d_W)\cap A^0_{r.d.}(P(L_0,L))$$
of $d_W$-closed functions which rapidly decay (hence the subscript ``r.d.'') along the lattice directions of $P(L_0,L)$, where
$P(L_0, L)$ is the {\em fiberwise geodesic path space} of the pair $(L_0,L)$ equipped with the $S^1$-valued {\em area function}
$$\cu{A}: P(L_0,L) \to S^1$$
and $d_W$ is the {\em Witten differential}.

Applying Witten-Morse theory \cite{Witten_Morse, CLM_Witten-Morse, Ma_thesis} to $P(L_0, L)$, we obtain the correspondence
$$\bigoplus_{p\in L_0\cap L}\bb{C}\cdot p \cong ker(d_W)\cap A^0_{r.d.}(P(L_0,L))$$
Then the SYZ transform (see Section \ref{sec:morphism})
$$\cu{F}:(A^{\bullet}(P(L_0,L)),d_W) \to (A^{0,\bullet}(\check{X},\check{L}),\dbar)$$
gives the identification
$$ker(d_W)\cap A^0_{r.d.}(P(L_0,L)) \cong H^0(\check{X},\check{L}).$$
Combining with \eqref{eqn:SYZ1}, which was obtained via another SYZ transform, we solve the quantization problem in the semi-flat setting:

\begin{theorem}[=Theorem \ref{thm:sf}]\label{thm:main_thm_1}
Let $\check{p}:\check{X}\to B$ be a semi-flat Lagrangian torus fibration over a compact complete special integral affine manifold $B$. Let $g$ be an integral (Definition \ref{def:integral_metric}) Hessian type metric on $B$. With respect to the prequantum line bundle $(\check{L}_g,\nabla_{\check{L}_g})$ associated to the metric $g$, the SYZ transform $\cu{F}$ induces a canonical isomorphism between $\Gamma_{P_{\bb{R}}}(\check{X},\check{L}_g)$ and $\Gamma_{P_{\bb{C}}}(\check{X},\check{L}_g)$
\end{theorem}

The prequantum line bundle $(\check{L}_g,\nabla_{\check{L}_g})$ here is explicitly given as the SYZ transform of the Lagrangian section
$$L_g:=\{(x,d\phi(x))\in X:x\in B\},$$
where $\phi$ is the local potential function for the Hessian type metric $g$.

Similar ideas can also be applied to projective toric manifolds:

\begin{theorem}[=Theorem \ref{thm:proj_toric}]\label{thm:main_thm_2}
Let $\check{p}:\check{X}\to B$ be the moment map of a projective toric manifold $\check{X}$. With respect to the prequantum line bundle $(\check{L}_{\phi},\nabla_{\check{L}_{\phi}})$, the SYZ transform $\cu{F}$ induces a canonical isomorphism between $\Gamma_{P_{\bb{R}}}(\check{X},\check{L}_{\phi})$ and $\Gamma_{P_{\bb{C}}}(\check{X},\check{L}_{\phi})$.
\end{theorem}

The Lagrangian section mirror to $(\check{L}_{\phi},\nabla_{\check{L}_{\phi}})$ here is again given by the graph of the differential of the K\"ahler potential $\phi$. However, in this case, due to degeneracy of the moment map fibers over the boundary $\partial B$ of the moment polytope $B$, not only Bohr-Sommerfeld Lagrangian fibers, but also {\em Bohr-Sommerfeld isotropic fibers}, would contribute. Hence we also need to consider the intersection points of $L_{\phi}$ with $L_0$ over the boundary $\partial B$ (or at infinity if one considers the Legendre dual $N_{\bb{R}}$ of $B$).

Detailed proofs of Theorems \ref{thm:main_thm_1} and \ref{thm:main_thm_2} will be given in Section \ref{ch:SYZ_geometric_quantization}.

\begin{remark}\label{rem:Maslov_Morse}
The Lagrangian sections $L_0$ and $L$ are graded, and as pointed out in \cite[Remark 13]{KS_HMS_torus_fibration}, the Maslov index of an intersection point of $L_0$ and $L$ coincides with the Morse index of the corresponding critical point of $\cu{A}$. Since the Lagrangian section $L$ satisfies the condition
$$Hess(\cu{A})(p)>0,$$
for every $p\in L_0\cap L$, so indeed all the intersection points of $L_0$ and $L$ are of Maslov index 0. 
It follows that
$$HF^{\bullet}(L_0,L)=HF^0(L_0,L)=\bigoplus_{p\in L_0\cap L}\bb{C}\cdot p.$$
Hence we indeed have $HF^0(L_0,L)\cong H^0(\check{X},\check{L})$ (as vector spaces), as predicted by the HMS conjecture.
\end{remark}

\subsection{Relation to other works}

The idea of applying mirror symmetry to geometric quantization (or vice versa) has appeared several times in the literature before our work and must be well-known among experts.

As far as we know, Andrei Tyurin was the first to suggest that the {\em numerical quantization problem}, namely, the dimension equality
$$\dim \Gamma_{P_{\bb{R}}}(\check{X},\check{L})=\dim \Gamma_{P_{\bb{C}}}(\check{X},\check{L}),$$
follows from the SYZ \cite{SYZ} and HMS \cite{HMS} conjectures, at least in the case of elliptic curves and algebraic K3 surfaces. Our SYZ transform $\mathcal{F}$ was called the {\em geometric Fourier transformation (GFT)} in his paper. Let us also point out that his calculations resembled those carried out in an earlier work \cite{Gross_SLag_I} of Gross.

The belief that the space of polarized sections for a complex polarization $P_\bb{C}$ is canonically isomorphic to that for the real polarization suggests that when a K\"ahler manifold $\check{X}$ admits a Lagrangian torus fibration (possibly with singular fibers), the space $H^0(\check{X}, \check{L})$ of holomorphic sections of a holomorphic line bundle $\check{L} \to \check{X}$ should have a {\em canonical basis} or a {\em theta basis}. This again was first stated by Tyurin in \cite{Tyurin}.
In his thesis \cite{Nohara}, Nohara reviewed this idea and, combined with an earlier work of Andersen \cite{Andersen}, reproved the numerical quantization equality for semi-flat Lagrangian torus fibrations over compact base.

In a series of works (\cite{GHS16} and upcoming works; see also the nice survey article \cite{GS_theta_function} for an overview), Gross, Hacking, Keel and Siebert show that a large class of varieties including Calabi-Yau manifolds carry {\em theta functions}. Instead of Lagrangian fibrations and Floer cohomologies (or their de-Rham versions such as the space $ker(d_W)\cap A^0_{r.d.}(P(L_0,L))$ that we use in this paper), they use tropical geometry and counts of so-called {\em broken lines} (tropical analogue of holomorphic disks) to construct theta functions and prove a strong form of Tyurin's conjecture. Their work was inspired by Tyurin's work, and the above heuristic argument for the existence of theta functions applying both the SYZ and HMS conjectures has actually already appeared on p.5 of \cite{GS_theta_function}. See also \cite[Section 8]{Kanazawa_GQ} for a nice review of these topics.

In \cite{BMN_abelian_varieties}, Baier, Mour\~ao and Nunes solved the quantization problem for abelian varieties and, together with Florentino, they solved the same problem for projective toric manifolds in \cite{Toric_geometric_quantization}. They considered a degenerating family of complex structures $\{\check{J}_s\}_{s\geq 0}$ approaching a large complex structure limit and which are compatible with a fixed symplectic structure $\check{\omega}$. As $s\to +\infty$, they proved that holomorphic sections of the pre-quantum line bundle would converge to distributional sections supported on Bohr-Sommerfeld fibers of the Lagrangian torus fibration, giving rise to the desired canonical isomorphism.

In our mirror symmetric approach, the same limiting process (here we use $\hbar>0$ and let $\hbar\to 0$) occurs when we apply Witten-Morse theory to the fiberwise geodesic path space $P(L_0,L)$ \cite{CLM_Witten-Morse, Ma_thesis}. If $s_{\hbar}$ is a holomorphic section of the prequantum line bundle $\check{L}_{\hbar}$, then its SYZ transform in $(A^{\bullet}(P(L_0,L)),d_W)$ (here $d_W$ depends on $\hbar$) converges to a sum of $\delta$-functions supported on the intersection points between $L_0$ and $L$ (or the critical points of the area function $\cu{A}$) as $\hbar\to 0$. Applying the SYZ transform (or Fourier transform) in the reverse direction transforms these $\delta$-functions to distributional sections supported on the Bohr-Sommerfeld fibers. In summary, under the SYZ transforms, the limiting process in \cite{BMN_abelian_varieties, Toric_geometric_quantization} is corresponding precisely to the Witten deformation, as illustrated in Figure \ref{fig:Witten_deformation}.

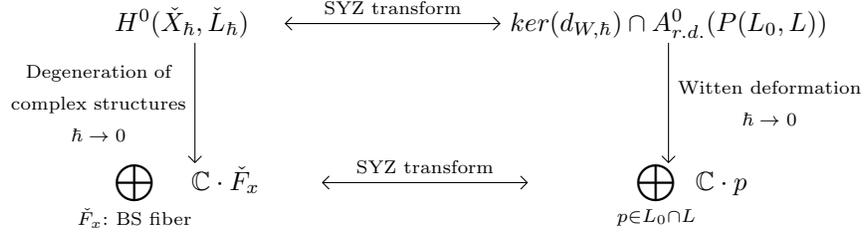
\begin{figure}
\begin{tikzpicture}
\matrix(m)[matrix of math nodes,
row sep=4.6em, column sep=6.8em,
text height=1.5ex, text depth=0.25ex]
{H^0(\check{X}_{\hbar},\check{L}_{\hbar})\quad &ker(d_{W,\hbar})\cap A_{r.d.}^0(P(L_0,L))\\
\displaystyle{\bigoplus_{\check{F}_x:\text{ BS fiber}}}\bb{C}\cdot \check{F}_x  \quad  \quad  & \quad \quad \quad \displaystyle{\bigoplus_{p\in L_0\cap L}}\bb{C}\cdot p \quad \quad \\};
\path[<->,font=\scriptsize,>=angle 90]
(m-1-1) edge node[auto] {$\text{SYZ transform}$} (m-1-2);
\path[->,font=\scriptsize,>=angle 90]
(m-1-1) edge node[left] {$\begin{matrix} \text{Degeneration of} \\ \text{complex structures} \\ \hbar \to 0 \end{matrix}$} (m-2-1);
\path[->,font=\scriptsize,>=angle 90]
(m-1-2) edge node[auto] {$\begin{matrix} \text{Witten deformation} \\ \hbar \to 0 \end{matrix}$} (m-2-2);
\path[<->,font=\scriptsize,>=angle 90]
(m-2-1) edge node[auto] {$\text{SYZ transform}$} (m-2-2);
\end{tikzpicture}
\caption{Degenerating family of complex structures vs. Witten deformation}\label{fig:Witten_deformation}
\end{figure}

\subsection*{Acknowledgment}
We are grateful to Conan Leung and Ziming Nikolas Ma for useful discussions. The first named author would like to thank Siye Wu for insightful comments and suggestions. The second named author would like to thank J{\o}rgen Ellegaard Andersen for his interest and comments on this work.

The work of K. C. was substantially supported by grants from the Research Grants Council of the Hong Kong Special Administrative Region, China (Project No. CUHK14314516 $\&$ CUHK14302617). The work of Y.-H. S. was supported by IBS-R003-D1.


\section{Semi-flat SYZ mirror symmetry}\label{sec:SYZ}

In this section, we review the constructions of the semi-flat mirror manifold and the SYZ transform of immersed Lagrangian multi-sections.

\subsection{Mirror construction in the semi-flat case}

Let $B$ be an $n$-dimensional integral affine manifold, that is, the transition functions of $B$ belongs to $GL(n,\bb{Z})\rtimes\bb{R}^n$, the group of $\bb{Z}$-affine linear map. Let $\Lambda\subset TB$ and $\Lambda^*\subset T^*B$ be the natural lattice bundles defined by the integral affine structure. More precisely, on a local affine chart $U\subset B$,
$$\Lambda(U):=\bigoplus_{j=1}^n\bb{Z}\cdot\pd{}{x^j},\quad \Lambda^*(U):=\bigoplus_{j=1}^n\bb{Z}\cdot dx^j,$$
where $(x^j)$ are affine coordinates of $U$. We set
$$X := T^*B/\Lambda^*, \quad \check{X} := TB/\Lambda.$$
Let $(y_j), (\check{y}^j)$ be fiber coordinates (which are dual to each other) of $X$ and $\check{X}$ respectively. Then $(x^j,y_j)$ and $(x^j,\check{y}^j)$ define a set of local coordinates on $T^*U/\Lambda^*\subset X$ and $TU/\Lambda\subset\check{X}$, respectively.

Equip $X$ with the standard symplectic structure
$$\omega_{\hbar}:=\hbar^{-1}\sum_jdy_j\wedge dx^j,$$
where $\hbar>0$ is a small real parameter. There is a natural almost complex structure $\check{J}_{\hbar}$ on $\check{X}$ given by
$$\check{J}_{\hbar}\left(\pd{}{x^j}\right)=-\hbar^{-1}\pd{}{\check{y}^j}\text{  and  }\check{J}_{\hbar}\left(\pd{}{\check{y}^j}\right)=\hbar\pd{}{x^j}.$$
It is easy to see that $\check{J}_{\hbar}$ is indeed integrable with local complex coordinates given by $z_j=\check{y}^j+i\hbar^{-1}x^j$. Hence $(\check{X},\check{J}_{\hbar})$ defines a complex manifold. To obtain Calabi-Yau structure, we need the following

\begin{definition}
An $n$-dimensional manifold $B$ is called a \emph{special integral affine manifold} if its transition functions sit in $SL(n,\bb{Z})\rtimes\bb{R}^n$.
\end{definition}

Hence if $B$ is an integral special affine manifold, then the canonical line bundle of $\check{X}$ is trivial. Indeed,
$$\check{\Omega}_{\hbar}:=dz_1\wedge\cdots\wedge dz_n$$
defines a global holomorphic volume form on $\check{X}$.

\begin{definition}
$(\check{X},\check{J}_{\hbar})$ is called the {\em SYZ mirror} of $(X,\omega_{\hbar})$.
\end{definition}

The $\hbar$-parameter gives us a family of symplectic manifolds. As $\hbar\to 0$, the symplectic volume of $(X,\omega_{\hbar})$ approaches infinity, which is the so-called {\em large volume limit} of the family $\{(X,\omega_{\hbar})\}_{\hbar>0}$.

Next, we equip K\"ahler structures on both $X$ and $\check{X}$.

\begin{definition}
A Riemannian metric on $B$ is said to be of \emph{Hessian type} if locally there is a smooth convex function $\phi$ such that the metric is locally given by
$$\sum_{j,k=1}^n\secpd{\phi}{x^j}{x^k}dx^j\otimes dx^k,$$
where $(x^j)$ are local affine coordinate of $B$.
\end{definition}

Suppose $g$ is a Hessian Riemannian metric on $B$. Then we can use it to construct natural complex structure on $X$ and symplectic structure on $\check{X}$ such that they compatible with the natural symplectic and complex structure on $X$ and $\check{X}$, respectively.

\begin{proposition}
Given a Hessian Riemannian metric on $B$, there is a natural complex structure $J_g$ on $X$ such that $(X,\omega_\hbar,J_g)$ is a K\"ahler manifold.
\end{proposition}
\begin{proof}
Define
$$dz_j=dy_j+i\sum_{k=1}^ng_{kj}dx^k.$$
Then one can easily check that $\{z_j\}$ define complex coordinates on $X$ and the complex structure $J_g$ is given by
$$J_g^*(dx^j)=\sum_{k=1}^ng^{jk}dy_k,$$
$$J_g^*(dy_j)=-\sum_{k=1}^ng_{kj}dx^k.$$
The symplectic form on $X$ is given by
$$\omega_\hbar=\hbar^{-1}\sum_{j=1}^ndy_j\wedge dx^j.$$
Hence
$$J_g^*\omega_\hbar=-\hbar^{-1}\sum_{j,k,l=1}^ng^{jl}g_{kj}dx^k\wedge dy_l=\hbar^{-1}\sum_{k=1}^ndy_k\wedge dx^k=\omega_\hbar.$$
and so
$$\omega_\hbar(\cdot,J_g(\cdot))=\sum_{j,k=1}^n\left(\hbar^{-2}g_{jk}dx^j\otimes dx^k+g^{jk}dy_j\otimes dy_k\right),$$
which is positive definite.
\end{proof}

Using this complex structure, we can also define a holomorphic volume form $\Omega_g$  on $X$ by
$$\Omega_g:=dz_1\wedge\cdots\wedge dz_n$$
when $B$ is in fact special. Then $(X,\omega_\hbar,J_g)$ is also a Calabi-Yau manifold.

Similarly, we have the following

\begin{proposition}\label{prop:sympl_B}
Given a Hessian Riemannain metric on $B$, there is a natural symplectic structure $\check{\omega}_g$ on $\check{X}$ such that $(\check{X},\check{\omega}_g,\check{J}_\hbar)$ is a K\"ahler manifold.
\end{proposition}
\begin{proof}
We define
$$\check{\omega}_g:=\sum_{j,k=1}^ng_{jk}d\check{y}^j\wedge dx^k.$$
Then $\check{\omega}_g,\check{J}_\hbar$ is compatible.
\end{proof}

As a conclusion, by fixing a Hessian type metric $g$ on $B$, we obtain two Calabi-Yau manifolds $(X,\omega_\hbar,J_g)$ and $(\check{X},\check{\omega}_g,\check{J}_\hbar)$.

\begin{remark}
Both the K\"ahler metrics $\omega_\hbar,\check{\omega}_g$ are Calabi-Yau metrics if and only if the local convex function $\phi$ satisfies the real Monge-Amp\`ere equation:
$$\det\left(Hess(\phi)\right)=\text{constant}.$$
\end{remark}

\subsection{SYZ transform of branes}

We follow \cite{AP, LYZ} to define the SYZ transform of a Lagrangian section in a semi-flat Lagrangian torus fibration.

Let $\mathcal{P}\to X\times_{B}\check{X}$ be the Poincar\'e line bundle. The total space is defined as the quotient
$$\cu{P}:=(T^*B\oplus TB)\times\bb{C}/\Lambda^*\oplus\Lambda.$$
The fiberwise action of $\Lambda^*\oplus\Lambda$ on $(T^*B\oplus TB)\times\bb{C}$ is given by
$$(\lambda,\check{\lambda})\cdot(y,\check{y},t):=(y+\lambda,\check{y}+\check{\lambda},e^{i\pi((y,\check{\lambda})-(\lambda,\check{y}))}\cdot t).$$
Define a connection $\nabla_{\mathcal{P}}$ on $\mathcal{P}$ by
$$\nabla_{\mathcal{P}}:=d+i\pi((y,d\check{y})-(\check{y},dy)).$$
The section $e^{i\pi(y,\check{y})}$ is invariant under the $\{0\}\oplus\Lambda$ action:
$$(0,\check{\lambda})\cdot(y,\check{y},t)=(y,\check{y}+\check{\lambda},e^{i\pi(y,\check{\lambda}+\check{y})}).$$
Hence it descends to a section $\check{1}$ on $T^*B\times_{B}\check{X}$. With respect to this frame, the connection $\nabla_{\mathcal{P}}$ can be written as
$$\nabla_{\mathcal{P}}=d+2\pi i(y,d\check{y}).$$
The remaining action of $\Lambda^*\oplus\{0\}$ then becomes
\begin{align*}
\lambda\cdot[(y,\check{y},e^{i\pi(y,\check{y})})]_{{\Lambda}} = e^{-2\pi i(\lambda,\check{y})}[y+\lambda,\check{y},e^{i\pi(y+\lambda,\check{y})}]_{\Lambda}.
\end{align*}

Let $\bb{L}=(L,\xi)$ be a Lagrangian section and $\cu{L}$ be a $U(1)$-local system on $L$ with holonomy $e^{2\pi ib}$, for some $b\in\bb{R}$. Define
$$\check{\bb{L}}_b:=(\pi_{\check{X}})_*((\xi\times id_{\check{X}})^*(\cu{P})\otimes\cu{L})).$$
where $\pi_{\check{X}}:L\times_{B}\check{X}\to\check{X}$ is the natural projection. The following proposition is standard.

\begin{proposition}\label{prop:sec_line}
The connection $\nabla_{\check{\bb{L}}_b}$ satisfies $(\nabla_{\check{\bb{L}}_b}^2)^{0,2}=0$ if and only if $\xi:L\to X$ is a Lagrangian section.
\end{proposition}

Hence $\check{\bb{L}}_b$ carries a natural holomorphic structure.

\begin{definition}
$(\check{\bb{L}}_b,\nabla_{\check{\bb{L}_b}})$ is called the {\em SYZ mirror bundle} of the A-brane $(\bb{L},\cu{L})$. For convenience, we just write the SYZ mirror of $(\bb{L},\cu{L})$ as $\check{\bb{L}}_b$ for short.
\end{definition}

\begin{remark}
It is known that every Lagrangian sections are graded Lagrangian immersions of the Calabi-Yau manifold $(X,\omega_\hbar,\Omega_g)$. For a proof, see \cite{KS_HMS_torus_fibration}.
\end{remark}

\subsection{SYZ transform of morphisms}\label{sec:morphism}

The exposition here follows Ma's PhD thesis \cite{Ma_thesis}.

Let $\bb{L}_1=(B,\xi_1),\bb{L}_2=(B,\xi_2)$ be two Lagrangian sections of the semi-flat fibration $p:X\to B$ and $\check{L}_1,\check{L}_2$ be the SYZ mirror bundles. Denote the images of $\xi_1$ and $\xi_2$ by $L_1$ and $L_2$ respectively. Let
$$P(L_1,L_2) := \coprod_{x\in B}\pi_1(p^{-1}(x);L_1,L_2)$$
be the fiberwise geodesic path space of $(L_1, L_2)$. It is a (possibly disconnected) covering space of $B$ via the natural projection map.

Let $p\in L_1\cap L_2$ and $x\in B$. Let $u:[0,1]\times[0,1]\to X$ be a smooth map such that for any $s,t\in[0,1]$,
$$u(s,0)\in L_1,\quad u(s,1)\in L_2,\quad u(0,t)=p,\quad u(1,t)\in p^{-1}(x).$$
Let $\til{P}(L_1,L_2)$ be set of all $(x,[u])$, where $x\in B$ and $[u]$ is the homotopy class of some $u$ that satisfies the above properties. There is a natural boundary map $\partial^+:\til{P}(L_1,L_2)\to P(L_1,L_2)$ given by
$$\partial^+:(x,[u])\mapsto(x,[u|_{\{1\}\times[0,1]}]).$$

Define the area function $\cu{A}:\til{P}(L_1,L_2)\to\bb{R}$ by
$$\cu{A}(x,[u])=\int_0^1\int_0^1u^*\omega_\hbar.$$
Since the boundary of $u$ lies in Lagrangian submanifolds, the integral depends only on the homotopy class of the map $u$.
In general, $\cu{A}$ may not be well-defined as a function on $P(L_1,L_2)$ but its differential
$$d\cu{A}(x,\gamma)=\dot{\gamma}\lrcorner\omega_\hbar=\hbar^{-1}\sum_{j=1}^n\left(\xi_j^{(2)}-\xi_j^{(1)}+m_j\right)dx^j,$$
where $(m_1,\dots,m_n)\in\bb{Z}^n$ corresponds to $\gamma\in\pi_1(p^{-1}(x);L_1,L_2)\cong\bb{Z}^n$, is a well-defined 1-form on $P(L_1,L_2)$. That is, there exists a 1-form $\theta$ on $P(L_1,L_2)$ such that $(\partial^+)^*\theta=d\cu{A}$. By abuse of notation, we use $d\cu{A}$ to stand for the 1-form $\theta$. But one should keep in mind that $\theta$ is, in general, not exact. Note that the critical points of $\cu{A}$ are precisely the intersection points of $L_1$ and $L_2$.

On $P(L_1,L_2)$, one can equip its space of differential forms $A^{\bullet}(P(L_1,L_2))$ with the Witten differential:
$$d_{W,\hbar}:=d+2\pi(\dot{\gamma}\lrcorner\omega_{\hbar})\wedge.$$
Following \cite{Ma_thesis} (see also \cite{Fukaya_asymptotic_analysis, KS_HMS_torus_fibration}), we define the SYZ transform
$\cu{F}:A^{\bullet}(P(L_1,L_2))\to A^{0,\bullet}(\check{X},\check{L}_1^*\otimes\check{L}_2)$ by
$$\cu{F}\left(\sum_I\alpha_I(x,m)dx^I\right):=\left(-\frac{\hbar}{2i}\right)^{|\alpha|}\sum_{m\in\bb{Z}^n}\sum_I\alpha_I(x,m)e^{2\pi i(m,\check{y})}d\bar{z}^I\otimes\check{1}_1^*\otimes\check{1}_2.$$
To avoid any convergence issues, we restrict ourselves to rapidly decay differential forms:

\begin{definition}
An element $\alpha\in A^{\bullet}(P(L_1,L_2))$ is said to {\em rapidly decay} if for any compact subset $K\subset B$, integer $k,l\geq 0$,
$$\lim_{|\dot\gamma|\to\infty}\sup_{x\in K}|\dot{\gamma}|^k|\nabla^l\alpha(x,\gamma)|=0,$$
where $\nabla^l$ stands for the $l$-times covariant derivative with respective to the affine connection $\nabla$ on $P(L_0,L)$ (induced from the one on $B$) and $|\cdot|$ is the norm induced by the Hessian-type metric $g$ on $B$. The space of all rapidly decay forms is denoted by $A^{\bullet}_{r.d.}(P(L_1,L_2))$ where the subscript ``r.d.'' stands for ``rapidly decay''.
\end{definition}

It is easy to see that $d_{W,\hbar}$ preserves $A^{\bullet}_{r.d.}(P(L_1,L_2))$. We can also define the inverse SYZ transform:
$$\cu{F}^{-1}\left(\sum_I\check{\alpha}_I(x,\check{y})d\bar{z}^I\otimes\check{1}_1^*\otimes\check{1}_2\right)_m:=\left(-\frac{\hbar}{2i}\right)^{-|\check{\alpha}|}\sum_I\left(\int_{\check{p}^{-1}(x)}\check{\alpha}(x,\check{y})e^{-2\pi i(m,\check{y})}d\check{y}\right)\otimes dx^I.$$
It is not hard to see that $\cu{F}$ and $\cu{F}^{-1}$ are indeed well-defined and inverse to each other. More importantly, $\cu{F}$ and $\cu{F}^{-1}$ exchange the Witten complex $(A^{\bullet}_{r.d.}(P(L_1,L_2)),d_{W,\hbar})$ with the Dolbeault complex $(A^{0,\bullet}(\check{X},\check{L}_1^*\otimes\check{L}_2),\dbar_{\hbar})$.

\begin{proposition}\label{prop:Fourier}
$\cu{F}:(A^{\bullet}_{r.d.}(P(L_1,L_2)),d_{W,\hbar})\to (A^{0,\bullet}(\check{X},\check{L}_1^*\otimes\check{L}_2),\dbar_{\hbar})$ is an isomorphism of differential graded vector spaces.
\end{proposition}
\begin{proof}
For $\alpha\in A^{\bullet}_{r.d.}(P(L_1,L_2))$, we compute that
\begin{align*}
  & \cu{F}(d_{W,\hbar}\alpha)\\
= & \left(-\frac{\hbar}{2i}\right)^{|\alpha|+1}\sum_{m\in\bb{Z}^n}\sum_I\sum_{k=1}^n\left(\pd{\alpha_I}{x^k}+2\pi\hbar^{-1}\left(\xi_1^k-\xi_2^k+m_k\right)\alpha_I\right)e^{2\pi i(m,\check{y})}d\bar{z}^k\wedge d\bar{z}^I\otimes\check{1}_1^*\otimes\check{1}_2\\
= & \left(-\frac{\hbar}{2i}\right)^{|\alpha|+1}\sum_{m\in\bb{Z}^n}\sum_I\sum_{k=1}^n\hbar^{-1}\left(\hbar\pd{}{x^k}\alpha_Ie^{2\pi i(m,\check{y})}-i\pd{}{\check{y}^k}\alpha_Ie^{2\pi i(m,\check{y})}\right)d\bar{z}^k\wedge d\bar{z}^I\otimes\check{1}_1^*\otimes\check{1}_2\\
  & \qquad + \left(-\frac{\hbar}{2i}\right)^{|\alpha|+1}\sum_{m\in\bb{Z}^n}\sum_I\sum_{k=1}^n2\pi\hbar^{-1}\left(\xi_1^k-\xi_2^k\right)\alpha_Ie^{2\pi i(m,\check{y})}d\bar{z}^k\wedge d\bar{z}^I\otimes\check{1}_1^*\otimes\check{1}_2\\
= & \left(\dbar+i\pi\sum_{k=1}^n(\xi_2^k-\xi_1^k)d\bar{z}^k\right)\cu{F}(\alpha)\\
= & \dbar_{\hbar}\cu{F}(\alpha),
\end{align*}
as desired.
\end{proof}

In order to connect ($S^1$-valued) Morse theory of $(P(L_1,L_2),\cu{A})$ with the de Rham model $(A^{\bullet}(P(L_1,L_2)),d_W)$, we apply the idea of Witten deformations, or Witten-Morse theory \cite{Witten_Morse}.

For a Morse function $f:M\to\bb{R}$ defined on a closed Riemannian manifold $(M,g)$, Witten considered the twisted differential
$$d_{f,\hbar}:=e^{-f/\hbar}de^{f/\hbar}=d+\hbar^{-1}df\wedge.$$
According to Witten-Morse theory, when $\hbar>0$ is small enough, there is a 1-1 correspondence between index $k$ critical points of $f$ and small eigenforms of degree $k$, concentrated at $p$:
\begin{equation*}
\xymatrix{
{Span_{\bb{R}}\left(Crit^k(f)\right)} \ar@{->}[r]^{\quad\sim} & {\Omega_{sm}^k(M,f)}
},\quad  p\mapsto\alpha_{p,\hbar},
\end{equation*}
where
$$\Omega_{sm}^k(M,f):=\left\{\alpha_{\hbar}\in\Omega^k(M):\Delta_{\hbar}\alpha_{\hbar}=\lambda_{\hbar}\alpha_{\hbar},\lim_{\hbar\to 0}\lambda_{\hbar}=0\right\},$$
and $\Delta_{\hbar}$ is the Witten Laplacian with respect to the metric $g$. The Morse cohomology of the pair $(M,f)$ is defined to be the cohomology of
$$(Span_{\bb{R}}\left(Crit^{\bullet}(f)\right),\delta),$$
where $\delta$ is defined by counting gradient flow lines from index $k$ critical points to index $k+1$ critical points. Witten also argued that the assignment
$$p\mapsto\alpha_{p,\hbar}$$
gives an isomorphism between the Morse cohomology and the cohomology of the Witten complex
$$(\Omega_{sm}^{\bullet}(M,f),d_{f,\hbar}),$$
which is nothing but the de Rham cohomology.

In Section \ref{ch:SYZ_geometric_quantization}, we will apply Witten-Morse theory to $(P(L_1,L_2),\cu{A})$ to solve the quantization problem.

\begin{remark}
On the A-side, morphisms between two Lagrangian submanifolds should be given by their Floer complex. But here we are using a twisted version of the de Rham complex as the morphism space. Let us clarify their (conjectural) relations.

As we mentioned in Remark \ref{rem:Maslov_Morse}, our Lagrangian sections $L_1,L_2$ are graded and the Maslov index of an intersection point between $L_1$ and $L_2$ coincides with the Morse index of the corresponding critical point of the area function $\cu{A}$ by \cite[Remark 13]{KS_HMS_torus_fibration}. It was further pointed out in \cite[Section 5.2]{KS_HMS_torus_fibration} that the Morse complex of $(P(L_1,L_2),\cu{A})$ (together with higher products) is an approximation of the Floer complex $CF^\bullet(L_1,L_2)$ when $\hbar$ is small (based on results of Floer \cite{Floer88} and Fukaya-Oh \cite{FO} in the case of cotangent bundles).

On the other hand, Witten's original work \cite{Witten_Morse} together with the recent work of Chan, Leung and Ma \cite{CLM_Witten-Morse} (see also \cite{Ma_thesis}) proved that the Morse complex is also an approximation of the twisted de Rham complex of $(P(L_1,L_2),\cu{A})$ when $\hbar$ is small.

In view of these results, we expect that the Floer complex $CF^\bullet(L_1,L_2)$ is quasi-isomorphic to the twisted de Rham complex of $(P(L_1,L_2),\cu{A})$. This is why we use the term ``morphisms'' in this section.

\end{remark}


\section{Geometric quantization and the quantization problem}

In this section, we apply SYZ transforms to solve the quantization problem for Lagrangian torus fibrations over compact complete integral affine manifolds and projective toric manifolds.

\subsection{Prequantum line bundles and polarizations}

We first recall the precise statement of the quantization problem.

\begin{definition}\label{def:prequantum}
A \emph{prequantum line bundle} on a K\"ahler manifold $(M, J, \omega)$ is a pair $(\cu{L},\nabla)$, where $\cu{L}$ is complex line bundle on $M$ and $\nabla$ is a complex integrable connection satisfying
$$\frac{i}{2\pi}F_{\nabla}=\omega.$$
In particular, $\cu{L}$ is an ample holomorphic line bundle.
\end{definition}

\begin{definition}\label{def:pola}
A \emph{polarization} is an integrable Lagrangian subbundle $P\subset T_{\bb{C}}M$ meaning that $rk(P)=n$ and $[P,P]\subset P$ and $\omega|_P=0$.
The space of all polarized sections of $\cu{L}$ is defined to be
$$\Gamma_P(M,\cu{L}) := \{s\in\overline{\Gamma}_{sm}(M,\cu{L}) : \nabla_vs = 0 \text{ for all } v\in P\},$$
where $\overline{\Gamma}_{sm}(M,\cu{L})$ denotes the $L^2$-completion of the space of smooth sections of $\cu{L}$.
\end{definition}

There are two typical choices of polarization, namely, the real and complex polarizations.

Given a Lagrangian fibration $f:M\to B$ of $M$ over an integral affine manifold $B$, we define
$$P_{\bb{R}}:=\ker(f_*:TM\to f^*TB)\otimes\bb{C}.$$

\begin{lemma}
$P_{\bb{R}}$ is a polarization.
\end{lemma}
\begin{proof}
Since $f:M\to B$ is a Lagrangian fibration, $P_{\bb{R}}$ is of rank $n$ and $\omega|_{P_{\bb{R}}}=0$. Let $(x^j)$ be coordinate of $B$ and $(y_j)$ be fiber coordinates. Sections of $P_{\bb{R}}$ is of form
$$\sum_{j=1}^ng^j(x,y)\pd{}{y_j},$$
and it follows that $[P_{\bb{R}},P_{\bb{R}}] \subset P_{\bb{R}}$.
\end{proof}

\begin{definition}
$P_{\bb{R}}$ is called the \emph{real polarization}.
\end{definition}

\begin{remark}
The terminology ``real polarization'' is justified by the property that $\overline{P}_{\bb{R}}=P_{\bb{R}}$.
\end{remark}

\begin{definition}\label{def:BS}
A submanifold $S$ of $M$ is called \emph{Bohr-Sommerfeld} if $\nabla|_S = d$.
\end{definition}

\begin{proposition}[\'Sniatycki \cite{cohomology_GQ}]\label{prop:BS_fiber_pola}
There is a canonical identification
$$\Gamma_{P_{\bb{R}}}(M,\cu{L}) \cong \bigoplus_{F_x:\text{BS fiber}}\bb{C}\cdot F_x.$$
\end{proposition}

Next, we define
$$P_{\bb{C}}:=T^{0,1}M=\{v\in T_{\bb{C}}M:\check{J}(v)=-iv\}.$$

\begin{lemma}
$P_{\bb{C}}$ is a polarization.
\end{lemma}
\begin{proof}
Clearly $rk(P_{\bb{C}})=n$. Since the complex structure of $M$ is integrable, we have $[P_{\bb{C}},P_{\bb{C}}]\subset P_{\bb{C}}$. Finally, $\omega$ is a K\"ahler form which is of $(1,1)$-type, so we have $\omega|_{P_{\bb{C}}}=0$.
\end{proof}

\begin{definition}
$P_{\bb{C}}$ is called the \emph{complex polarization}.
\end{definition}

The space of complex polarized sections is nothing but the space of holomorphic sections of $\cu{L}$:
$$\Gamma_{P_{\bb{C}}}(M,\cu{L}) = H^0(M, \cu{L}).$$

In geometric quantization, physicists ask whether a quantum theory is independent of the choice of polarizations. In other words, for two given polarizations $P,P'$, they are asking if the spaces of polarized sections are canonically isomorphic to each other. This is commonly known as the \emph{quantization problem}. In particular, for a K\"ahler manifold admitting a Lagrangian fibration, we expect that the real and complex polarizations should give canonically isomorphic spaces of polarized sections.

\section{Main results}

We prove that, in the case of semi-flat Lagrangian torus fibrations over compact complete bases and projective toric manifolds, the spaces of real and complex polarized sections are canonically isomorphic via the SYZ transforms defined in Section \ref{sec:SYZ}. The proofs also use Witten-Morse theory on the fiberwise geodesic path spaces.

\subsection{Semi-flat Lagrangian torus fibrations}\label{ch:SYZ_geometric_quantization}

Throughout this section, $B$ is assumed to be an $n$-dimensional compact special integral affine manifold without boundary. We also assume $\hbar=1$.

Choose a Hessian type metric $g$ on $B$ and let $\check{X}:=TB/\Lambda$, which is equipped with the natural complex structure $\check{J}$. By Proposition \ref{prop:sympl_B}, $(\check{X},\check{J},\check{\omega}_g)$ is a K\"ahler manifold (in fact a Calabi-Yau manifold).

\begin{definition}\label{def:integral_metric}
A Hessian type metric $g$ is said to be \emph{integral} if there exists an open cover $\{U\}$ of $B$ such that on each non-empty overlap $U\cap V$, the potential functions $\phi_U:U\to\bb{R}$, $\phi_V:V\to\bb{R}$ of $g$ satisfy
$$\phi_U(x_U(x_V))-\phi_V(x_V)=\inner{m_{UV},x_V}+a_{UV},$$
for some $m_{UV}\in\bb{Z}^n$ and $a_{UV}\in\bb{R}$.
\end{definition}

Let $g$ be an integral Hessian type metric on $B$. Then we can define a Lagrangian section $L_g$ of the dual fibration $p:X\to B$ by
$$L_g:=\{(x,d\phi(x))\in X:x\in B\}.$$
The integrality condition on $g$ ensures that $L_g$ is independent of the potential function $\phi$. Hence it is a well-defined Lagrangian section of $p:X\to B$.

The SYZ transform $\check{L}_g$ of $L_g$ has connection
$$\nabla_{\check{L}_g}=d+2\pi i\sum_{j=1}^n\pd{\phi}{x^j}d\check{y}_j.$$
Moreover, the curvature of $\nabla_{\check{L}_g}$ satisfies
$$\frac{i}{2\pi}F_{\nabla_{\check{L}_g}}=-\sum_{j,k}\frac{\partial^2\phi}{\partial x^j\partial x^k}dx^j\wedge d\check{y}_k=\check{\omega}_g.$$
Hence we obtain the following

\begin{proposition}
The SYZ transform $(\check{L}_g,\nabla_{\check{L}_g})$ of the Lagrangian section $L_g$ is a prequantum line bundle on $\check{X}$.
\end{proposition}

A key observation is given by the following

\begin{proposition}\label{prop:BS_fiber_int_pt}
With respect to the prequantum line bundle $(\check{L}_g,\nabla_{\check{L}_g})$, there is a 1-1 correspondence between Bohr-Sommerfeld fibers of $\check{p}:\check{X}\to B$ and intersection points of $L_0$ and $L_g$.
\end{proposition}
\begin{proof}
The connection $\nabla_{\check{L}_g}$ is trivial on a fiber $\check{F}_x$ of $\check{p}:\check{X}\to B$ if and only if $d\phi(x)\in\Lambda_x^*$, that is, a point in $L_0\cap L_g$.
\end{proof}

In order to prove the next lemma, we need to introduce the following

\begin{definition}\label{def:complete}
An affine manifold is said to be \emph{complete} if its universal cover is diffeomorphic to $\bb{R}^n$ as affine manifold.
\end{definition}

\begin{remark}
There is a famous conjecture by Markus \cite{Markus_conjecture}, stating that any compact special affine manifold is complete. Hence if Markus' conjecture is true, then the completeness condition is automatic. since we have assumed $B$ to be a compact special integral affine manifold.
\end{remark}

\begin{lemma}\label{lem:decomposition}
Suppose the base $B$ is complete. Then we have a decomposition
$$P(L_0,L_g)=P\amalg\coprod_{p\in L_0\cap L_g}P_p(L_0,L_g),$$
where $P_p(L_0,L_g)$ is the connected component of $p$. Moreover, each $P_p(L_0,L_g)$ is contractible.
\end{lemma}
\begin{proof}
Suppose $P_p(L_0,L_g)\cap P_q(L_0,L_g)\neq\phi$ with $p\neq q$. Since $P_p(L_0,L_g)$ and $P_q(L_0,L_g)$ are connected components, we have $P_p(L_0,L_g)=P_q(L_0,L_g)$. Since $B$ is complete and $P_p(L_0,L_g)$ is a covering of $B$, the universal cover $\pi:\til{P}_p\to P_p(L_0,L_g)$ of $P_p(L_0,L_g)$ is affinely diffeomorphic to $\bb{R}^n$. Hence the connection $\nabla$ pulls back to a connection on $\til{P}_p$ which is gauge equivalent to the trivial connection $d$. Since any two points in $\bb{R}^n$ can be join by a geodesic with respect to the trivial connection, so is $P_p(L_0,L_g)$. Let $(l(s),\gamma_{\l(s)})$ be a geodesic connecting $p$ and $q$. Then
$$u(s,t):= (l(s),\gamma_{l(s)}(t))$$
is a disk in $X$ with the properties
$$u(0,t)=p,\quad u(1,t)=q,\quad u(s,0)\in L_0,\quad u(s,1)\in L_g.$$
Consider the function $f:[0,1]\to\bb{R}$ defined by
$$f(s):=g_{X}\left(\pd{u}{s}(s,1),\pd{u}{t}(s,1)\right).$$
In an affine chart, it can be written as
$$f(s)=\sum_{j=1}^n\dot{l}^j(s)\left(\pd{\phi}{x^j}(l(s))+m_j\right).$$
Since $f(0)=f(1)=0$, there exists $s_0\in (0,1)$ such that $f'(s_0)=0$. In terms of the local expression of $f$, we have
$$\sum_{j=1}^n\ddot{l}^j(s_0)\left(\pd{u}{x^j}(l(s_0))+m_j\right)+\sum_{j,k=1}^n\frac{\partial^2\phi}{\partial x^j\partial x^k}(l(s_0))\dot{l}^j(s_0)\dot{l}^k(s_0)=0.$$
The path $l:[0,1]\to B$ is also a geodesic with respect to the flat connection $\nabla$, so $\ddot{l}^j(s_0)=0$ for all $j$. Hence
$$\sum_{j,k=1}^n\frac{\partial^2\phi}{\partial x^j\partial x^k}(l(s_0))\dot{l}^j(s_0)\dot{l}^k(s_0)=0.$$
But $\dot{l}(s_0)$ is a nonzero vector, this contradicts positivity of $Hess(\phi)$. Therefore, $P_p(L_0,L_g)\cap P_q(L_0,L_g)=\phi$ whenever $p\neq q$.

Let $[\alpha]$ be a homotopy class of loops of $P_p(L_0,L_g)$ based at $p$. Let $\til{p}\in\til{P}_p$ be a lift of $p$. Suppose $[\alpha]\neq 0$. Choose any loop $\alpha\in[\alpha]$. By the path lifting property, there is a lift $\til{\alpha}:[0,1]\to\til{P}_p$ of $\alpha$ such that $\til{\alpha}(0)=\til{p}$. Since $[\alpha]\neq 0$, $\til{p}':=\til{\alpha}(1)\neq\til{p}$. Choose a geodesic $c\subset\til{P}_p$ with starting point $\til{p}$ and end point $\til{p}'\neq\til{p}$. Then $[\pi\circ c]=[\alpha]$. In particular, $\pi\circ c$ is a non-constant geodesic loop in $P_p(L_0,L_g)$, based at $p$. Let's write $\pi\circ c$ as
$$(\pi\circ c)(s)=(l(s),\gamma_{l(s)}).$$
Then
$$u(s,t):=(l(s),\gamma_{l(s)}(t))$$
is a disk in $X$ satisfying
$$u(s,0)\in L_0,\quad u(s,1)\in L_g,\quad u(0,t)=p,\quad u(1,t)=p.$$
Again, we can consider the function
$$f(s):=g_{X}\left(\pd{u}{s}(s,1),\pd{u}{t}(s,1)\right).$$
Since $l(s)$ is an non-constant geodesic, the same proof above applies to conclude there exists some $s_0$ such that
$$Hess(\phi)(l(s_0))(\dot{l}(s_0),\dot{l}(s_0))=0.$$
Hence we must have $\alpha=0$, that is, $P_p(L_0,L_g)$ is simply connected. In particular, $P_p(L_0,L_g)\cong\bb{R}^n$ as affine manifold.
\end{proof}

In order to rule out the solution on $P$, we need the following

\begin{lemma}\label{lem:max}
On a component of $P(L_0,L_g)$, any rapidly decay positive (resp. negative) functions have a maximum (resp. minimum).
\end{lemma}
\begin{proof}
Let $P_0$ be a component of $P(L_0,L_g)$. Then $P_0$ is a covering of $B$. If $P_0$ is a finite covering of $B$, then we are done since $B$ is compact. Suppose $P_0$ is an infinite covering of $B$. Then there in an infinite set $S\subset\bb{Z}^n$ such that $P_0$ is a $S$-covering of $B$. Now, since $B$ is compact, we can choose a finite open covering $\cu{U}=\{U\}$ for $B$. We assume each $U$ is evenly covered by open sets in $P_0$. Let $\cu{V}=\{V\}$ be a pre-compact refinement of $\cu{U}$, that is, $\overline{V}\subset U$ and $\overline{V}$ is compact. Then by identifying $\overline{V}$ with a compact subset of $P_0$, we have
$$P_0=\bigcup_{m\in  S}\bigcup_{V\in\cu{V}}\overline{V}\times\{m\}.$$
Let $f:P_0\to\bb{R}$ be a rapidly decay positive function. Then for any $\epsilon>0$ and $V\in\cu{V}$, there exists $N_V$ such that for $|m|\geq N_V$,
$$\sup_{x\in\overline{V}}f(x,m)<\epsilon.$$
Let $N:=\max_{V\in\cu{V}}N_V$. Then for $|m|\geq N$, the above inequality holds. Restricting $f$ to the compact set
$$\bigcup_{|m|<N}\bigcup_{V\in\cu{V}}\overline{V}\times\{m\},$$
$f$ achieves a maximum in it. Since $f$ is bounded on $P_0$, $\sup_{P_0}f$ exists and is positive. By taking $\epsilon$ smaller then $\sup_{P_0} f$, we conclude that $f$ has a global maximum.
\end{proof}

Let $\cu{A}_p$ be the primitive of $d\cu{A}$ defined on the component $P_p(L_0,L_g)$ with $\cu{A}_p(p)=0$, where $d\cu{A}(p)=0$. Then we obtain

\begin{proposition}\label{prop:Witten_Morse}
Suppose $B$ is complete. Then the assignment
$$A:p\mapsto e^{-2\pi\cu{A}_p}$$
gives an identification between $\bigoplus_{p\in L_0\cap L_g}\bb{C}\cdot p$ and $\ker(d_W)\cap A^0_{r.d.}(P(L_0,L_g))$. In particular, by composing with the SYZ transform $\cu{F}$, we obtain the identification of morphism spaces:
\begin{equation*}
\xymatrix{
{\cu{F}\circ A:\displaystyle{\bigoplus_{p\in L_0\cap L_g}}\bb{C}\cdot p} \ar@{->}[r]^{\quad\quad\sim} & {H^0(\check{X},\check{L}_g)}
}.
\end{equation*}
\end{proposition}
\begin{proof}
We claim that non-trivial rapidly decay solution to $d_Wf=0$ has a critical point and this implies $d_Wf=0$ has no non-trivial solution on $P$. To prove our claim, suppose $f$ is a non-trivial rapidly decay solution defined on a component $P_0$ of $P$. By Lemma \ref{lem:max}, $f$ achieves its global maximum or minimum at some point $p_0\in P_0$. Hence $df(p_0)=0$ and so
$$f(p_0)d\cu{A}(p_0)=0.$$
By replacing $f$ by $-f$, we simply assume $f(p_0)$ is the maximum of $f$. If $f(p_0)\neq 0$, then $d\cu{A}(p_0)=0$. If $f(p_0)=0$, then $f\leq 0$. We can then find $p_1$ such that $f(p_1)$ is the global minimum of $f$. Hence $df(p_1)=0$ and so
$$f(p_1)d\cu{A}(p_1)=0.$$
Since $f$ is non-trivial, $f(p_1)\neq 0$. Again, we have $d\cu{A}(p_1)=0$. But $P$ contains no critical point of $\cu{A}$. Therefore, $f$ must be trivial.

Next, we prove that $e^{-2\pi\cu{A}_p}$ is rapidly decay. Note that the equation $d_Wf=0$ has a unique solution (up to a constant multiple) because $d\cu{A}$ is exact on $P_p(L_0,L_g)$. Hence for each critical point $p$ of $\cu{A}$, the vector space $\ker(d_W)\cap A^0(P_p(L_0,L_g))$ is 1-dimensional. Now, for two critical point $p,q$, both $P_p(L_0,L_g)$ and $P_q(L_0,L_g)$ are universal covering spaces of $B$. Hence there exists a (non-unique) diffeomorphism $\psi_{qp}:P_p(L_0,L_g)\to P_q(L_0,L_g)$ covering the identity $id_B$. In particular, it preserves the 1-form $d\cu{A}$, that is,
$$\psi^*_{qp}d\cu{A}_q=d\cu{A}_p.$$
Hence $e^{-2\pi\cu{A}_q\circ\psi_{qp}}$ is a constant multiple of $e^{-2\pi\cu{A}_p}$. Now, the inverse SYZ transform and Lemma \ref{lem:decomposition} give the identification:
$$\cu{F}^{-1}:H^0(\check{X},\check{L}_g)\cong\ker(d_W)\cap A^0_{r.d.}(P(L_0,L_g))=\bigoplus_{p\in L_0\cap L_g}\ker(d_W)\cap A^0_{r.d.}(P_p(L_0,L_g)).$$
This implies $\dim(\ker(d_W)\cap A^0_{r.d.}(P_p(L_0,L_g)))=1$ for some $p$. As rapidly decay condition is preserved by any fiber preserving diffeomorphism, we have, for any $q$,
$$e^{-2\pi\cu{A}_q}\propto e^{-2\pi\cu{A}_p\circ\psi^{-1}_{qp}}$$
is also rapidly decay.

Therefore, we conclude that $A$ gives the isomorphism
$$\bigoplus_{p\in L_0\cap L_g}\bb{C}\cdot p\cong\ker(d_W)\cap A^0_{r.d.}(P(L_0,L_g)).$$
\end{proof}

Therefore, the function $e^{-2\pi\cu{A}_p}$ in Proposition \ref{prop:Witten_Morse} is the Witten-deformation of the intersection point $p\in L_0\cap L_g$.

Proposition \ref{prop:Witten_Morse} can be applied to abelian varieties:

\begin{example}
Let $\Omega$ be a positive definite symmetric $n\times n$ matrix with real entries. Let $\check{X}:=\bb{C}^n/\bb{Z}^n\oplus i\Omega\bb{Z}^n$ be the abelian variety with period $\Omega$. Complex coordinates are given by $z^j:=\check{y}^j+ix^j$. The mirror of $\check{X}$ is $X:=\bb{R}^{2n}/\bb{Z}^{2n}$ equipped with the symplectic structure
$$\omega=\sum_{j,k=1}^n\Omega_{jk}dy_j\wedge dx^k.$$
Let $Q$ be any positive definite symmetric integral matrix with the property that
$$Q\Omega=\Omega Q.$$
Then
$$L_Q:=\{(x,Qx)\in X:x\in\bb{R}^n/\bb{Z}^n\}.$$
is a Lagrangian section of $p:X\to \bb{R}^n/\bb{Z}^n$. Let $x_1,\dots,x_N\in B$ such that $Qx_k\in\bb{Z}^n$ for all $k=1,\dots,N$. In this case, the geodesic path space is a disjoint union of $N$ copies of $\bb{R}^n$ and the Witten differential on a component $P_{(x_k,0)}(L_0,L_Q)$ is given by
$$d+2\pi\sum_{j,l,r=1}^nQ_{jl}\left(x^l-x_k^l+m_l\right)\Omega_{jr}dx^r.$$
Hence the function $A_k:=A(x_k,0)$ is given by
\begin{align*}
A_k(x,m) & = \exp\left(-2\pi\int_{x_k}^{x+m}\int_0^{Q(u-x_k)}\omega\right)\\
& = \exp\left(-\pi(x-x_k+m)^tQ\Omega(x-x_k+m)\right)\\
& = \exp\left(-\pi(x-x_k)^tQ\Omega(x-x_k)\right)\exp\left(-2\pi\left(m^tQ\Omega(x-x_k)+\frac{1}{2}m^tQ\Omega m\right)\right)
\end{align*}
Note that $Q\Omega$ is still positive definite since $[Q,\Omega]=0$. The local holomorphic frame $\check{e}_1$ for $\check{L}_Q$ is given by
$$\check{e}_1=e^{-\pi(x-x_k)^tQ\Omega(x-x_k)}\check{1}_1.$$
The SYZ transform of $A_k$ is given by
\begin{align*}
\cu{F}(A_k)(x,\check{y}) & = \sum_{m\in\bb{Z}^n}e^{-2\pi m^tQ\Omega m}e^{-2\pi(Q\Omega m,x-x_k)}e^{2\pi i(Q m,\check{y})}\otimes\check{e}_1\\
& = \sum_{m\in\bb{Z}^n}e^{-2\pi m^tQ\Omega m}e^{2\pi i(m,Q(\check{y}+i\Omega(x-x_k))}\otimes\check{e}_1
\end{align*}
and its pullback via the covering map $\bb{C}^n\to\check{X}$ is the Riemann theta function
$$\theta\begin{bmatrix}0\\-iQ\Omega x_k\end{bmatrix}(iQ\Omega,Qz_{\Omega})=\sum_{m\in\bb{Z}^n}e^{-2\pi m^tQ\Omega m}e^{2\pi i(m,Qz_{\Omega}-iQ\Omega x_k)}$$
on $\bb{C}^n$, where $z_{\Omega}:=\check{y}+i\Omega x$. In fact, the conditions $[Q,\Omega]=0$ and $Q>0$ are equivalent to the Riemann bilinear relations:
$$\begin{pmatrix}
  I_n & i\Omega
\end{pmatrix}
\begin{pmatrix}
  O_n & Q\\
  -Q & O_n
\end{pmatrix}^{-1}
\begin{pmatrix}
  I_n \\
  i\Omega
\end{pmatrix}=0,\quad
-i\begin{pmatrix}
  I_n & i\Omega
\end{pmatrix}
\begin{pmatrix}
  O_n & Q\\
  -Q & O_n
\end{pmatrix}^{-1}
\begin{pmatrix}
  I_n \\
  -i\Omega
\end{pmatrix}>0,
$$
which gives the ample line bundle $\check{L}_Q$ on $\check{X}$ and the space of global holomorphic sections of $\check{L}_Q$ is generated by the theta functions $\cu{F}(A_k)$, $k=1,\dots,N$.
\end{example}

Combining Propositions \ref{prop:BS_fiber_pola} and \ref{prop:Witten_Morse}, we obtain our first main result:
\begin{theorem}\label{thm:sf}
Let $\check{p}:\check{X}\to B$ be a semi-flat Lagrangian torus fibration over a compact complete special integral affine manifold $B$. Let $g$ be an integral Hessian type metric on $B$. With respect to the prequantum line bundle $(\check{L}_g,\nabla_{\check{L}_g})$, the SYZ transform $\cu{F}$ induces a canonical isomorphism between the spaces of real and complex polarized sections.
\end{theorem}

We expect that Proposition \ref{prop:Witten_Morse} and Theorem \ref{thm:sf} also hold for non-compact but complete base $B$ if we impose suitable growth conditions on the space of holomorphic sections $H^0(\check{X},\check{L}_g)$, as in the following example.

\begin{example}
Let $B=\bb{R}^n$ and $\check{X}=(\bb{C}^{\times})^n$. We have the natural torus fibration $\check{p}:\check{X}\to B$. The mirror of $\check{X}$ is given by $X=(\bb{C}^{\times})^n$. Consider the Lagrangian section
$$L_1:=\{(x,[x])\in X:x\in B\}$$
of the dual fibration $p:X\to B$. The mirror line bundle is isomorphic to the trivial line bundle $\cu{O}_{\check{X}}$ since $\check{X}$ is affine. The set of intersection points between $L_1$ and the zero section $L_0$ is given by
$$\{(k,0)\in X:k\in\bb{Z}\}.$$
On the component $P_{(k,0)}(L_0,L_1)$, the function $A_k:=A((k,0))$ is given by
$$A_k(x)=\exp\left(-\pi(x-k)^2\right).$$
If we let $1:=\cu{F}(A_0)$, then the SYZ transform $\cu{F}(A_k)$ of $A_k$ is proportional to monomial $z^k$ on $\check{X}=(\bb{C}^{\times})^n$. Hence if we restrict our attention to $H^0_{poly}(\check{X},\cu{O}_{\check{X}})$, the space of all holomorphic functions on $\check{X}$ that have polynomial growth (which are just Laurent polynomials), then we obtain the isomorphisms
$$\Gamma_{P_{\bb{R}}}(\check{X},\check{L}_1)\cong\bigoplus_{p\in L_0\cap L_1}\bb{C}\cdot p\cong H^0_{poly}(\check{X},\cu{O}_{\check{X}}),$$
via Witten-Morse theory and the SYZ transform.
\end{example}

\subsection{Projective toric manifolds}

We now turn to the quantization problem for projective toric manifolds. Let us first recall some basic facts in toric geometry.

Let $N\cong\bb{Z}^n$ be a lattice of rank $n$ and set
$$N_{\bb{R}}:=N\otimes_{\bb{Z}}\bb{R},\quad M:=Hom_{\bb{Z}}(N,\bb{Z}),\quad M_{\bb{R}}:=M\otimes_{\bb{Z}}\bb{R}.$$
Let $\Sigma$ be a fan with primitive generators $v_1,\dots,v_d\in N$. Let $\check{X}:=X_{\Sigma}$ be the toric variety associated to $\Sigma$. We assume $\check{X}$ is smooth and projective. The Picard group of $\check{X}$ has an explicit description as follows. Let $\iota:M\to \bb{Z}^d$ be given by
$$\iota:u\mapsto\left(\inner{u,v_1},\dots,\inner{u,v_d}\right).$$
The assignment
$$[a]\mapsto\cu{L}_{[a]}:=\cu{O}\left(\sum_{j=1}^da_jD_j\right),$$
where $D_j\subset\check{X}$ is the toric divisor corresponds to the ray $v_j$, gives the identification
$$Pic(\check{X})\cong\bb{Z}^d/\iota(M).$$

Since $\check{X}$ is assumed to be projective, we can take $\cu{L}_{[\lambda]}$ to be an ample line bundle on $\check{X}$. It is a well-known fact in toric geometry that such a line bundle is in fact very ample (see \cite{Fulton_book}). Hence it determines an embedding $i:\check{X}\hookrightarrow\bb{P}^N$ into some projective space $\bb{P}^N$. Let $\check{\omega}:= i^*\omega_{FS}$, where $\omega_{FS}$ is the Fubini-Study metric on $\bb{P}^N$. The dense torus in $\check{X}$ can be identified with $TN_{\bb{R}}/N$. Denote the coordinates on $N_{\bb{R}}$ by $\xi^j$ and the induced complex coordinates $z^j:=\check{y}^j+i\xi^j$ on $TN_{\bb{R}}/N$. It is well-known (see \cite{Guillemin_book}) that on $TN_{\bb{R}}/N$, the K\"ahler form $\check{\omega}$ can be written as
$$\check{\omega}:=2i\partial\dbar\phi,\quad\phi(\xi):=\frac{1}{4\pi}\log\left(\sum_{u\in B\cap M}c_ue^{4\pi\inner{u,\xi}}\right).$$
Here $c_u\geq 0$ are constants that depend on the embedding $i:\check{X}\hookrightarrow\bb{P}^N$.

The $(\bb{C}^{\times})^n$-action on $\check{X}$ restricted to a Hamiltonian $T^n$-action on $(\check{X},\check{\omega})$ such that the moment map $\check{p}:\check{X}\to M_{\bb{R}}$ has image
$$B:=\{x\in M_{\bb{R}}:\inner{x,v_k}+\lambda_j\geq 0,k=1,\dots,d\},$$
which is a convex polytope in the vector space $M_{\bb{R}}$. The interior fibers of $\check{p}:\check{X}\to B$ are special Lagrangian tori with respect to the following holomorphic volume form
$$\check{\Omega}:=\frac{dz^1}{z^1}\wedge\cdots\wedge\frac{dz^n}{z^n},$$
which has a simple pole along the toric divisors. The fibers of $\check{p}$ get degenerated to isotropic tori on the boundary $\partial B$. Moreover, the space of holomorphic sections of $\cu{L}_{[\lambda]}$ can be identified with the vector space spanned by the lattice points in $B$, that is $M\cap B$.

\begin{remark}
One can identify $N_{\bb{R}}$ with $\mathring{B}$, the interior of $B$, via the differential $d\phi:N_{\bb{R}}\to\mathring{B}$. Moreover, if we let $\pi:TN_{\bb{R}}/N\to N_{\bb{R}}$ be the natural projection, the moment map $\check{p}$ factors through $N_{\bb{R}}$ as $\check{p}=d\phi\circ\pi$.
\end{remark}

Before going into geometric quantization, let us recall mirror symmetry for a projective toric manifold.
The mirror of $\check{X}$ is given by the Landau-Ginzburg model $(X,W)$, where $X:=T^*\mathring{B}/\Lambda^*$ and $W$ is a holomorphic function on $X$, called the superpotential. Explicitly, we have
$$X=\left\{(z_1,\dots,z_n)\in (\bb{C}^{\times})^n:|e^{-\lambda_j}z^{v_j}|<1,\forall j=1,\dots,d\right\}$$
and
$$W(z_1,\dots,z_n):=\sum_{j=1}^de^{-\lambda_{j}}z^{v_{j}}.$$
Here, the complex coordinate $z_j$ is given by
$$z_j:=e^{2\pi i(y_j+i\phi_j)},$$
where $\phi_j=\pd{\phi}{\xi^j}$. Hence we can identify $X$ with $T^*N_{\bb{R}}/M$ via
$$(y_1+i\phi_1,y_2+i\phi_2)\mapsto (y_1+i\xi_1,y_2+i\xi_2).$$
Hence $X$ admits a special Lagrangian torus fibration over $\mathring{B}$ with respect to the symplectic form
$$\sum_{j=1}^ndy_j\wedge d\xi^j$$
and the holomorphic volume form
$$\Omega=\frac{dz_1}{z_1}\wedge\cdots\wedge\frac{dz_n}{z_n}.$$

In \cite{Chan09}, the first named author of this paper defined a certain class of Lagrangian sections of the dual fibration $p:X\to\mathring{B}$ whose SYZ transform are in $Pic(\check{X})$. Moreover, every such line bundle carries a natural $T^n$-invariant Hermitian metric. Let's recall this class of objects. We identify $\mathring{B}$ with $N_{\bb{R}}$ by $d\phi$.

\begin{definition}
Let $a:=(a_1,\dots,a_d)\in\bb{Z}^d$. A Lagrangian section $\til{L}$ of $p:T^*N_{\bb{R}}\to N_{\bb{R}}$ is said to satisfy \emph{Condition $(*_a)$} if there is a $C^2$-potential function $g:N_{\bb{R}}\to\bb{R}$ of $\til{L}$ satisfies the following conditions: For any top dimensional cone $\sigma\in\Sigma(n)$, without loss of generality, assume $\sigma$ is generated by $v_1,\dots,v_n$ and let $\xi(t)=t_1v_1+\cdots+t_nv_n$, for $t_j\in\bb{R}$, we have
\begin{itemize}
\item [1.] For any $j=1,\dots,n$, the limits
\begin{align*}
&\lim_{t_j\to-\infty}2e^{-4\pi t_j}(\inner{dg(\xi(t)),v_j}+a_j),\\
&\lim_{t_j\to-\infty}e^{-4\pi t_j}\text{Hess}(g(\xi(t)))(v_j,v_j)
\end{align*}
exist and equal.
\item [2.] For any $j,k,l=1,\dots,n$, the limit
$$\lim_{t_l\to-\infty}\text{Hess}(g(\xi(t)))(v_j,v_k)$$
exists.
\item [3.] For any $j,k=1,\dots,n$ with $j\neq k$, we have
$$\lim_{t_j\to-\infty}e^{-2\pi(t_j+t_k)}\text{Hess}(g(\xi(t)))(v_j,v_k)=0$$
or
$$\lim_{t_k\to-\infty}e^{-2\pi(t_j+t_k)}\text{Hess}(g(\xi(t)))(v_j,v_k)=0.$$
\end{itemize}
Let $[a]\in\bb{Z}^d/\iota(M)$. A Lagrangian section of $p:X\to N_{\bb{R}}$ is said to satisfy \emph{Condition $(*_{[a]})$} if for some lift $\til{L}\subset T^*N_{\bb{R}}$ of $L$, $\til{L}$ satisfies Condition $(*_a)$ for some representative $a$ of $[a]$.
\end{definition}

\begin{remark}
There is a slight difference between our definition and that given in \cite{Chan09}. We require all the exponential terms in the limits consist of a fact of $2\pi$. This difference is due to our choice of complex coordinates on $X$ being $e^{2\pi i(y_j+i\phi_j)}$, while in \cite{Chan09}, the author used $e^{\phi_j+iy_j}$ instead.
\end{remark}

The main result in \cite{Chan09} is the following

\begin{theorem}[Chan \cite{Chan09}]\label{thm:Chan}
Fix $a\in\bb{Z}^d$. The SYZ transform gives a 1-1 correspondence between Lagrangian sections satisfying Condition $(*_{[a]})$ and $T^n$-invariant Hermitian metrics on $\cu{L}_{[a]}$ of $C^2$-class.
\end{theorem}

In particular, we know that whenever a Lagrangian section satisfies Condition $(*_{[a]})$, then its SYZ transform is isomorphic to the holomorphic line bundle $\cu{L}_{[a]}$.

Now we are ready to work on the quantization problem for the projective toric manifold $\check{X}$. As the fibers of the moment map $\check{p}$ degenerates to isotropic tori on the boundary $\partial B$, it is natural consider Bohr-Sommerfeld isotropic submanifolds. It is already known that Bohr-Sommerfeld isotropic fibers are precisely those fibers above the lattice points of $B$ (see \cite{Toric_geometric_quantization}), and these lattice points also correspond to holomorphic sections of the associated line bundle of the polytope $B$, the equivalence of real and complex polarization is almost trivial. Here we give a mirror symmetric interpretation of this equivalence.

We consider the Lagrangian section
$$L_{\phi}:=\{(\xi,d\phi(\xi))\in X:\xi\in N_{\bb{R}}\}.$$
Applying the SYZ transform, we obtain a holomorphic line bundle $\check{L}_{\phi}$ on the dense torus $TN_{\bb{R}}/N\subset\check{X}$ together with a connection $\nabla_{\check{L}_{\phi}}$. The following proposition shows that the pair $(\check{L}_{\phi},\nabla_{\check{L}_{\phi}})$ in fact extends to a holomorphic prequantum line bundle on $\check{X}$.

\begin{proposition}\label{lem:same_bundle}
The line bundle $\check{L}_{\phi}$ extends to a holomorphic line bundle on $\check{X}$ and is isomorphic to the line bundle $\cu{L}_{[\lambda]}$. Moreover, $(\check{L}_{\phi},\nabla_{\check{L}_{\phi}})$ defines a prequantum line bundle on $\check{X}$.
\end{proposition}
\begin{proof}
Fix a top dimensional cone $\sigma\in\Sigma$. By renaming, we assume $\sigma$ is generated by $v_1,\dots,v_n$. Let $\xi(t)=\sum_{j=1}^nt_jv_j$. We show that
$$\lim_{t_j\to-\infty}\inner{d\phi(\xi(t)),v_j}=-\lambda_j.$$
We compute that
$$\inner{d\phi(\xi(t)),v_j}=\frac{\sum_{u\in B\cap M}\inner{u,v_j}c_ue^{4\pi\inner{u,\xi}}}{\sum_{u\in B\cap M}c_ue^{4\pi\inner{u,\xi}}}.$$
Then
$$\lim_{t_j\to-\infty}\inner{d\phi(\xi(t)),v_j}=\min_{u\in B\cap M}\inner{u,v_j}=-\lambda_j.$$
It is not hard to see that the Lagrangian $L_{\phi}$ satisfies Condition $(*_{[\lambda]})$. Hence by Theorem \ref{thm:Chan}, $(\check{L}_{\phi},\nabla_{\check{L}_{\phi}})$ extends to $\check{X}$ and $\check{L}_{\phi}\cong\cu{L}_{[\lambda]}$ as holomorphic line bundle.

For the prequantum condition, recall that the connection $\nabla_{\check{L}_{\phi}}$ is given by
$$\nabla_{\check{L}_{\phi}}=d+2\pi i\sum_{j=1}^n\pd{\phi}{\xi^j}d\check{y}^j.$$
Hence
$$\frac{i}{2\pi}F_{\nabla_{\check{L}_{\phi}}}=-\sum_{j=1}^n\frac{\partial^2\phi}{\partial\xi^k\partial\xi^j}d\xi^k\wedge d\check{y}^j=\check{\omega}.$$
This completes the proof of the proposition.
\end{proof}

As we have mentioned, the spaces of real and complex polarized sections are respectively given by
\begin{align*}
\Gamma_{P_{\bb{R}}}(\check{X},\check{L}_{\phi}) & \cong \bigoplus_{\check{F}_x:\text{BS fiber}}\bb{C}\cdot\check{F}_x,\\
\Gamma_{P_{\bb{C}}}(\check{X},\check{L}_{\phi}) & = H^0(\check{X},\check{L}_{\phi}).
\end{align*}
Note that the direct sum is taken over all Bohr-Sommerfeld isotropic fibers.

As the semi-flat case, we would like to establish a correspondence between $L_0\cap L_{\phi}$ and Bohr-Sommerfeld fibers. However, this is not true if we only consider interior intersections. What we need is to include the intersections at infinity:

\begin{definition}\label{def:int_equiv}
A Lagrangian section $L$ of $p:X\to B$ is said to be \emph{intersecting the zero section $L_0$ at infinity} if there exists a lift $\til{L}\subset T^*N_{\bb{R}}$ of $L$, a potential function $g:N_{\bb{R}}\to\bb{R}$ of $\til{L}$, a ray $\xi:[0,\infty)\to N_{\bb{R}}$ such that the limit
$$\lim_{t\to\infty}dg(\xi(t))$$
exists in $M$. Fix a potential $g$. Let
$$\cu{R}_{\bb{Z}}(L) := \left\{\xi:[0,\infty)\to N_{\bb{R}}\Big|\lim_{t\to\infty}dg(\xi(t))\text{ exists in }M\right\}.$$
Two rays $\xi_1,\xi_2\in \cu{R}_{\bb{Z}}(L)$ are said to be \emph{integrally equivalent} if
$$\lim_{t\to\infty}dg(\xi_1(t))=\lim_{t\to\infty}dg(\xi_2(t)).$$
Denote the set of all equivalence classes of such rays by $\overline{L_0\cap L}$.
\end{definition}

\begin{remark}
Interior intersection points are precisely those equivalence classes that can be represented (uniquely) by constant path.
\end{remark}

In view of Section \ref{ch:SYZ_geometric_quantization}, one would like to consider the fiberwise geodesic path space $P(L_0,L_{\phi})$. However, there are not enough critical points if we only consider interior intersections. To recall the information coming from infinity, we consider
$$\cu{R}(L_{\phi}):=\left\{\xi:[0,\infty)\to N_{\bb{R}}\Big|\lim_{t\to\infty}d\phi(\xi(t))\text{ exists in }M_{\bb{R}}\right\}.$$
Similar to Definition \ref{def:int_equiv}, we say two rays $\xi_1,\xi_2$ are equivalent if
$$\lim_{t\to\infty}dg(\xi_1(t))=\lim_{t\to\infty}dg(\xi_2(t)).$$
Denote $\overline{N_{\bb{R}}}:=\cu{R}(L_{\phi})/\sim$ and $\overline{P(L_0,L_{\phi})}:=\overline{N_{\bb{R}}}\times M$.

\begin{remark}
There is a natural identification between $\overline{N_{\bb{R}}}$ and the polytope $B$, given by
$$[\xi]\mapsto\lim_{t\to\infty}d\phi(\xi(t)).$$
Moreover, $\overline{L_0\cap L_{\phi}}$ is mapped to $M\cap B$, the set of lattice points of $B$.
See also \cite{Fulton_book}.
\end{remark}

The original path space $P(L_0,L_{\phi})$ can be identified with $N_{\bb{R}}\times M$ via
$$\left(\xi,s(d\phi(\xi)+m)\right)\mapsto (\xi,m).$$
Hence $P(L_0,L_{\phi})$ sits inside $\overline{P(L_0,L_{\phi})}$ naturally by
$$(\xi,m)\mapsto ([\xi],m),$$
where $[\xi]$ is the equivalence class of the constant path $\xi$. Note that the interior intersection points are of form
$$\left([d\phi^{-1}(u)],-u\right),$$
with $u\in\mathring{B}\cap M$. For the intersections at infinity, they are of the form
$$\left([\xi],-u\right),$$
where $\xi\in\cu{R}_{\bb{Z}}(L_{\phi})$ and $u\in\partial B\cap M$ is nothing but the limit
$$\lim_{t\to \infty}d\phi(\xi(t)).$$

Next, we study the Morse theory on the fiberwise geodesic path space $P(L_0,L_{\phi})$. However, not all functions on $P(L_0,L_{\phi})$ can be transformed to global sections of $\check{L}_{\phi}$.

Let us recall the the coordinate charts associated to a vertex of $B$. Let $V(B)$ be the vertex set of $B$. Consider the coordinate chart $\check{p}^{-1}(B_v)$, where
$$B_v:=\{v\}\cup\mathring{B}\cup\bigcup_{F:\text{ face contains }v}\mathring{F}$$
and $v\in V(B)$. Let
$$l_k(x):=\inner{x,v_k}+\lambda_k,\quad k=1,\dots,d.$$
By smoothness, without loss of generality, we assume the vertex $v$ is given by
$$l_1(x)=\cdots=l_n(x)=0.$$
Let $A=(A_{jk})\in GL(n,\bb{Z})$ be the differential of the affine map
$$(x_1,\dots,x_n)\mapsto (l_1(x),\dots,l_n(x)).$$
The gluing map $F_v:TN_{\bb{R}}/N\to\check{p}^{-1}(B_v)$ of $\check{p}^{-1}(B_v)\cong\bb{C}^n$ to $TN_{\bb{R}}/N\cong(\bb{C}^{\times})^n$ is given by
$$F_v:\left(e^{2\pi i(\check{y}^j+i\xi^j)}\right)_{j=1,\dots,n}\mapsto\left(e^{2\pi i\sum_{k=1}^n{^tA}^{jk}(\check{y}^k+i\xi^k)}\right)_{j=1,\dots,n},$$
where $({^tA}^{jk})$ is the transpose of the inverse of $A$. We write
$$\xi_v^j:=\sum_{k=1}^n{^tA}^{jk}\xi^k,\quad \check{y}^j_v:=\sum_{k=1}^n{^tA}^{jk}\check{y}^k.$$
The complex coordinates
$$Z_v^j:=X_v^j+i\check{Y}_v^j:=e^{2\pi iz^j_v}:=e^{2\pi i(\check{y}^j_v+i\xi^j_v)}$$
extend to $\check{p}^{-1}(B_v)\cong\bb{C}^n$.

Recall the SYZ transform of a function $f:P(L_0,L)\to\bb{C}$ is given by
$$\cu{F}(f)=\sum_{m\in\bb{Z}^n}\alpha_I(\xi,m)e^{2\pi i(m,\check{y})}\otimes\check{1}_L.$$
On $\check{p}^{-1}(B_v)$, there is an unitary frame $\check{1}_v$ such that
$$\check{1}_v=e^{2\pi i(v,\check{y})}e^{2\pi i(\lambda,\check{y})}\check{1}_L,$$
where $v\in V(B)$ and $\lambda\in\bb{Z}^n$ corresponds to the difference between the choices of lift of $L$ to $TN_{\bb{R}}$ on the two charts $\check{p}^{-1}(\mathring{B})$, $\check{p}^{-1}(\mathring{B}_v)$. In terms of the coordinates $\xi_v^j,\check{y}_v^j$, we have
$$\cu{F}(f)=\sum_{m\in\bb{Z}^n}f({}^tA\xi_v,A^{-1}m+\lambda)e^{2\pi i(m,\check{y}_v)}\otimes\check{1}_v.$$

\begin{lemma}\label{lem:ext_fun}
Let $F:(\bb{C}^{\times})^n\to\bb{C}$ be a smooth function. Equip $\bb{C}^n$ with the standard flat metric $\bar{g}$. Then $F$ extends to a smooth function on $\bb{C}^n$ if and only if for any pre-compact open subset $U\subset\bb{C}^n$ and $j\in\bb{Z}_{\geq 0}$, the covariant derivatives $\nabla^jF$ are bounded on $U\cap(\bb{C}^{\times})^n$.
\end{lemma}
\begin{proof}
We cover $\bb{C}^n$ by pre-compact open ball $\cu{U}:=\{U\}$. Fix $U\in\cu{U}$. Since $U$ is connected and the divisor $\bigcup_{j=1}^n\{z_j=0\}$ is of real co-dimension 2, for each two points $z_0,z_1\in U\cap (\bb{C}^{\times})^n$, we can choose a regular path $\gamma:[0,1]\to U\cap(\bb{C}^{\times})^n$ joining $z_0$ to $z_1$. By reparametrizing $\gamma$, we can assume $\gamma$ has constant speed $|\dot\gamma|_{\bar{g}}=|z_1-z_0|$, the distance between $z_0$ and $z_1$. It follows from the mean value theorem that
$$|\nabla^jF(z_1)-\nabla^jF(z_0)|_{\bar{g}}\leq M_U|z_1-z_0|,$$
for some constant $M_U>0$, depending on the open ball $U$. This implies $\nabla^jF$ is uniformly continuous on $U\cap(\bb{C}^{\times})^n$. Hence we obtain an unique continuous extension of $\nabla^jF$ to $\overline{U}$. If $U,V$ are two overlapping pre-compact open sets in $\bb{C}^n$, by uniqueness, the extensions coincide on $\overline{U}\cap\overline{V}$. Hence $\nabla^jF$ extends to a global continuous section on $\bb{C}^n$. The converse is trivial.
\end{proof}

\begin{remark}
The result of Lemma \ref{lem:ext_fun} dose not depend on the chosen metric on $\bb{C}^n$. We use flat metrics only for the sake of convenience.
\end{remark}

For a function $F:p^{-1}(\mathring{B}_v)\cong(\bb{C}^{\times})^n\to\bb{C}$, write
$$\nabla^kF=\sum_{|I|+|J|=k}a_{IJ}(\xi_v,\check{y}_v)d\xi^I_v\otimes d\check{y}^J_v+\sum_{|I|+|J|=k}b_{IJ}(\xi_v,\check{y}_v)d\check{y}^I_v\otimes d\xi^J_v,$$
where we put $b_{IJ}\equiv 0$ if $I=\phi$ or $J=\phi$. For each multi-sets $I,J$, let $P_{IJ},Q_{IJ}$ be the linear differential operator so that
\begin{align*}
P_{IJ}F & = a_{IJ},\\
Q_{IJ}F & = b_{IJ}.
\end{align*}
In terms of the coordinates $\xi_v^j,\check{y}_v^j$, the flat metric $\bar{g}$ reads
$$\bar{g}=4\pi^2\sum_{j=1}^ne^{-4\pi\xi_v^j}\left(d\xi_v^j\otimes d\xi_v^j+d\check{y}_v^j\otimes d\check{y}_v^j\right)$$
and the only non-zero Christoffel symbols are $\Gamma_{jj}^j=-2\pi$, which are constants. We see that the coefficients of $P_{IJ},Q_{IJ}$ are in fact constants. Let's denote $\hat{P}_{IJ},\hat{Q}_{IJ}$ the differential operator on $N_{\bb{R}}\times\{m\}$ by replacing $\pd{}{\check{y}^j_v}$ by multiplication by $2\pi im_j$.

\begin{definition}
Let $L\subset X$ be a Lagrangian section that satisfies Condition $(*_{[a]})$ for some $a\in\bb{Z}^d$. A function $f\in A^{0}_{r.d.}(P(L_0,L))$ is said to be \emph{weakly extendable} if for any vertex $v\in V(B)$, multi-sets $I,J$ and pre-compact open subset $U\subset B_v$, there exists constant $M_{IJ,U}>0$ such that
\begin{align*}
\left|\hat{P}_{IJ}f({}^tA\xi_v,A^{-1}m+\lambda)\right|\leq & M_{IJ,U}\prod_{i\in I}e^{-2\pi\xi^i_v}\prod_{j\in J}e^{-2\pi\xi_v^j},\\
\left|\hat{Q}_{IJ}f({}^tA\xi_v,A^{-1}m+\lambda)\right|\leq & M_{IJ,U}\prod_{i\in I}e^{-2\pi\xi^i_v}\prod_{j\in J}e^{-2\pi\xi_v^j},
\end{align*}
for all $m\in M$, $\xi_v\in d\phi^{-1}(U)$ and $\check{y}_v\in\check{T}^n$.
\end{definition}

In particular, if $f$ is weakly extendable, we see that it is bounded on $N_{\bb{R}}=d\phi^{-1}(\mathring{B})$.
The terminology ``weakly extendable'' is justified by the following

\begin{lemma}\label{lem:ext_sec}
If $\cu{F}(f)$ extends to a smooth section of $\check{L}$ on $\check{X}$, then $f$ is weakly extendable.
\end{lemma}
\begin{proof}
Since $\check{1}_v$ is a section defined on $\check{p}^{-1}(B_v)$, in order to obtain a global section of $\check{L}$, the functions
$$F(\xi_v,\check{y}_v):=\sum_{m\in\bb{Z}^n}f_m(\xi_v)e^{2\pi i(m,\check{y}_v)}:=\sum_{m\in\bb{Z}^n}f({}^tA\xi_v,A^{-1}m+\lambda)e^{2\pi i(m,\check{y}_v)}$$
need to be extended to smooth functions defined on $\check{p}^{-1}(B_v)$. Equivalently, by Lemma \ref{lem:ext_fun}, for every pre-compact subset $U\subset\bb{C}^n$, they have bounded covariant derivatives on $U\cap(\bb{C}^{\times})^n$.

We have
\begin{align*}
|\nabla^kF|^2_{\bar{g}} & = \sum_{|I|+|J|=k}\left(|P_{IJ}F|^2+|Q_{IJ}F|^2\right)\prod_{i\in I}e^{4\pi\xi^i_v}\prod_{j\in J}e^{4\pi\xi^j_v},\\
P_{IJ}F & = \sum_{m\in\bb{Z}^n}\hat{P}_{IJ}f_m(\xi_v)e^{2\pi i(m,\check{y}_v)},\\
Q_{IJ}F & = \sum_{m\in\bb{Z}^n}\hat{Q}_{IJ}f_m(\xi_v)e^{2\pi i(m,\check{y}_v)}.
\end{align*}
Now, suppose $|\nabla^kF|_{\bar{g}}\leq M_U$ for some constant $M_U$. Then
\begin{align*}
|P_{IJ}F|\leq & M_{IJ,U}\prod_{i\in I}e^{-2\pi\xi^i_v}\prod_{j\in J}e^{-2\pi\xi^j_v},\\
|Q_{IJ}F|\leq & M_{IJ,U}\prod_{i\in I}e^{-2\pi\xi^i_v}\prod_{j\in J}e^{-2\pi\xi^j_v},
\end{align*}
for some constants $M_{IJ,U}>0$. Hence by absorbing the constants, we have
\begin{align*}
|\hat{P}_{IJ}f_m|\leq & \int_{\check{T}^n}|P_IJF|d\check{y}\leq M_{IJ,U}\prod_{i\in I}e^{-2\pi\xi^i_v}\prod_{j\in J}e^{-2\pi\xi^j_v},\\
|\hat{Q}_{IJ}f_m|\leq & \int_{\check{T}^n}|Q_IJF|d\check{y}\leq M_{IJ,U}\prod_{i\in I}e^{-2\pi\xi^i_v}\prod_{j\in J}e^{-2\pi\xi^j_v}.
\end{align*}
Thus, $f$ is weakly extendable.
\end{proof}

\begin{remark}
We expect there should be a stronger condition on $f$ so that $\cu{F}(f)$ extends to a smooth section of $\check{L}$.
\end{remark}

Now, we return to the Lagrangian section $L_{\phi}$. Lemma \ref{lem:ext_sec} shows that not all functions on $P(L_0,L_{\phi})$ can be transformed into global sections of $\check{L}_{\phi}$. Next, we deal with holomorphic sections. Via the inverse SYZ transform, we need to consider the solution to $d_Wf=0$. We have the following

\begin{lemma}\label{lem:ext_holo}
$f:N_{\bb{R}}\times\{m\}\to\bb{R}$ is an non-trivial weakly extendable solution to $d_Wf=0$ on the component $N_{\bb{R}}\times\{m\}$ of $P(L_0,L_{\phi})\cong N_{\bb{R}}\times M$ if and only if $f$ is proportional to
$$f_{u}(\xi)=e^{-2\pi\phi(\xi)}e^{2\pi\inner{u,\xi}}$$
where $u\in M\cap B$.
\end{lemma}
\begin{proof}
On each component $P_m:=N_{\bb{R}}\times\{m\}$ of $P(L_0,L_{\phi})$, solution to $d_Wf=0$ is proportional to
$$f_m(\xi)=e^{-2\pi\phi(\xi)}e^{2\pi\inner{m,\xi}}.$$
We claim that $m\in M\cap B$ if and only if $f_m$ weakly extendable.

For $u\in B\cap M$, we apply the SYZ transform to $f_u$ to obtain
$$\cu{F}(f_u)(z)=e^{-2\pi i\inner{u,z}}\otimes\check{e}_{\phi},$$
which is nothing but the character  corresponds to $u$. It is well-known in toric geometry that $\cu{F}(f_u)$ extends to a holomorphic section of $\cu{L}_{[\lambda]}\cong\check{L}_{\phi}$. Hence $f_u$ is weakly extendable by Lemma \ref{lem:ext_sec}.

Conversely, note that $\phi(\xi)-\inner{m,\xi}$ is convex, we have
$$\phi(\xi)-\inner{m,\xi}\geq\phi(\xi')-\inner{m,\xi'}+\inner{d\phi(\xi')-m,\xi-\xi'},$$
for any $\xi,\xi'\in N_{\bb{R}}$. As $m\notin M\cap B$, there exists $k\in\{1,\dots,d\}$ such that
$$\inner{m,v_k}+\lambda_k<0.$$
Hence
$$\inner{d\phi(\xi')-m,v_k}>0.$$
for all $\xi'\in N_{\bb{R}}$. In terms of the polytope coordinates, this means
$$\inner{x-m,v_k}>0,$$
for all $x\in B$. Since $B$ is compact, there exists $\delta>0$ such that
$$\inner{x-m,v_k}>\delta,$$
for all $x\in B$. Equivalently,
$$\inner{d\phi(\xi')-m,v_k}>\delta,$$
for all $\xi'\in N_{\bb{R}}$. Putting $\xi=0$ and $\xi'=tv_k$, for $t>0$, we get
$$\phi(tv_k)-\inner{m,tv_k}<\phi(0)-t\delta.$$
Hence
$$\lim_{t\to\infty}\left(\phi(tv_k)-\inner{m,tv_k}\right)=-\infty$$
and hence $f_m$ is unbounded in the $v_k$-direction. Therefore, $f_m$ cannot be weakly extendable for $m\notin B\cap M$.
\end{proof}

As before, the function $f_u$ in Lemma \ref{lem:ext_holo} is regarded as the Witten-deformation of the intersection point $[\xi(t)]\in\overline{L_0\cap L_{\phi}}$, with
$$\lim_{t\to\infty}\xi(t)=u.$$

We are now ready to give another proof of the quantization problem via SYZ transforms; the original proof can be found in \cite{Toric_geometric_quantization}.

\begin{theorem}\label{thm:proj_toric}
Let $\check{p}:\check{X}\to B$ be the moment map of the projective toric manifold $\check{X}$. With respect to the prequantum line bundle $(\check{L}_{\phi},\nabla_{\check{L}_{\phi}})$, the SYZ transform $\cu{F}$ induces a canonical isomorphism between the space of real and complex polarized sections.
\end{theorem}
\begin{proof}
In terms of the coordinates $(\xi,\check{y})\in TN_{\bb{R}}/N$, the connection $\nabla_{\check{L}_{\phi}}$ is given by
$$\nabla_{\check{L}_{\phi}}=d+2\pi i\sum_{j=1}^n\pd{\phi}{\xi^j}d\check{y}^j.$$
Suppose $\xi(t)$ is a ray so that
$$\lim_{t\to\infty}d\phi(\xi(t))=u,$$
for some $u\in M$ ($u\in M\cap B$ in fact). If $u\in\mathring{B}$, then
$$\lim_{t\to\infty}\nabla_{\check{L}_{\phi}}|_{\check{F}_{\xi(t)}}=d+2\pi i\sum_{j=1}^nu_jd\check{y}^j,$$
which is equivalent to the trivial connection on the Lagrangian fiber $\check{F}_u$. If $u\in\partial B$, we choose $v\in V(B)$ such that $u\in B_v$. Since $u\in\partial B_v$, there exists subset $S_u\subset\{1,\dots,n\}$ such that the fiber $\check{F}_u$ sits inside the divisor
$$\left\{\prod_{j\in S_u}Z_j=0\right\}\subset\bb{C}^n\cong\check{p}^{-1}(B_v).$$
Recall we have a frame $\check{1}_v$ on $\check{p}^{-1}(B_v)$ and coordinates $\xi_v^j,\check{y}_v^j$ on $\check{p}^{-1}(B_v)$ such that
\begin{align*}
\check{1}_v & = e^{2\pi i(v,\check{y})}e^{2\pi i(\lambda,\check{y})}\check{1}_{\phi},\\
\xi_v^j & = \sum_{k=1}^n{^tA}^{jk}\xi^k,\\
\check{y}^j_v & = \sum_{k=1}^n{^tA}^{jk}\check{y}^k.
\end{align*}
Then with respect to $\check{1}_v$ and the coordinates $\xi_v^j,\check{y}_v^j$, we have
$$\nabla_{\check{L}_{\phi}}=d+2\pi i\sum_{j,k=1}^nA_{jk}\left(v_k+\lambda_k+\pd{\phi}{\xi^k}\right)d\check{y}_v^j.$$
Hence
$$\lim_{t\to\infty}\nabla_{\check{L}_{\phi}}|_{\check{F}_{\xi(t)}}=d+2\pi i\sum_{j\notin S_u}(A(v+\lambda+u))_jd\check{y}^j_v.$$
Since $A(v+\lambda+u)\in\bb{Z}^n$, it is equivalent to the trivial connection on the isotropic fiber $\check{F}_u$. Hence we obtain the identification
$$\bigoplus_{\check{F}_x:\text{BS fiber}}\bb{C}\cdot\check{F}_x\cong\overline{L_0\cap L_{\phi}}.$$
On a component of the fiberwise geodesic path space $P(L_0,L_{\phi})\cong N_{\bb{R}}\times M$, by Lemma \ref{lem:ext_holo}, bounded non-trivial solution to $d_Wf=0$ is proportional to
$$f_u(\xi)=e^{-2\pi\phi(\xi)}e^{2\pi \inner{u,\xi}},$$
which we can regard it as a function on $P(L_0,L_{\phi})$ by setting
$$
f_{u}(\xi,m):=
\left\{
\begin{array}{ll}
e^{-2\pi\phi(\xi)}e^{2\pi \inner{u,\xi}} &\text{ if }m=-u\\
0 &\text{ if } m\neq -u
\end{array}
\right.
$$
As we have seen in the proof of Lemma \ref{lem:ext_holo}, the SYZ transform of $f_u$ is given by the character
$$\cu{F}(f_u)(z)=e^{-2\pi i\inner{u,z}}\otimes\check{e}_{\phi}$$
corresponds to $u$, which extends to a global holomorphic section of $\cu{L}_{[\lambda]}$. Since $\cu{L}_{[\lambda]}\cong\check{L}_g$ as holomorphic line bundles by Lemma \ref{lem:same_bundle} and any holomorphic section of $\cu{L}_{[\lambda]}$ is a linear combination of these holomorphic sections, the result follows.
\end{proof}

\bibliographystyle{amsplain}
\bibliography{geometry}

\providecommand{\bysame}{\leavevmode\hbox to3em{\hrulefill}\thinspace}
\providecommand{\MR}{\relax\ifhmode\unskip\space\fi MR }
\providecommand{\MRhref}[2]{%
  \href{http://www.ams.org/mathscinet-getitem?mr=#1}{#2}
}
\providecommand{\href}[2]{#2}
\begin{thebibliography}{10}

\bibitem{Andersen}
J.~E. Andersen, \emph{Geometric quantization of symplectic manifolds with
  respect to reducible non-negative polarizations}, Comm. Math. Phys.
  \textbf{183} (1997), no.~2, 401--421. \MR{1461965}

\bibitem{AP}
D.~Arinkin and A.~Polishchuk, \emph{Fukaya category and {F}ourier transform},
  Winter {S}chool on {M}irror {S}ymmetry, {V}ector {B}undles and {L}agrangian
  {S}ubmanifolds ({C}ambridge, {MA}, 1999), AMS/IP Stud. Adv. Math., vol.~23,
  Amer. Math. Soc., Providence, RI, 2001, pp.~261--274. \MR{1876073}

\bibitem{Toric_geometric_quantization}
T.~Baier, C.~Florentino, J.~M. Mour\~ao, and J.~P. Nunes, \emph{Toric
  {K}\"ahler metrics seen from infinity, quantization and compact tropical
  amoebas}, J. Differential Geom. \textbf{89} (2011), no.~3, 411--454.
  \MR{2879247}

\bibitem{BMN_abelian_varieties}
T.~Baier, J.~M. Mour\~ao, and J.~P. Nunes, \emph{Quantization of abelian
  varieties: distributional sections and the transition from {K}\"ahler to real
  polarizations}, J. Funct. Anal. \textbf{258} (2010), no.~10, 3388--3412.
  \MR{2601622}

\bibitem{Chan09}
K.~Chan, \emph{Holomorphic line bundles on projective toric manifolds from
  {L}agrangian sections of their mirrors by {SYZ} transformations}, Int. Math.
  Res. Not. IMRN (2009), no.~24, 4686--4708. \MR{2564372 (2011k:53125)}

\bibitem{CLM_Witten-Morse}
K.-L. Chan, N.~C. Leung, and Z.~N. Ma, \emph{Fukaya’s conjecture on
  {W}itten’s twisted {$A_\infty$} structure}, J. Differential Geom., to
  appear, \href{https://arxiv.org/abs/1401.5867}{arXiv:1401.5867}.

\bibitem{Floer88}
A.~Floer, \emph{Morse theory for {L}agrangian intersections}, J. Differential
  Geom. \textbf{28} (1988), no.~3, 513--547. \MR{965228}

\bibitem{Fukaya_asymptotic_analysis}
K.~Fukaya, \emph{Multivalued {M}orse theory, asymptotic analysis and mirror
  symmetry}, Graphs and patterns in mathematics and theoretical physics, Proc.
  Sympos. Pure Math., vol.~73, Amer. Math. Soc., Providence, RI, 2005,
  pp.~205--278. \MR{2131017}

\bibitem{FO}
K.~Fukaya and Y.-G. Oh, \emph{Zero-loop open strings in the cotangent bundle
  and {M}orse homotopy}, Asian J. Math. \textbf{1} (1997), no.~1, 96--180.
  \MR{1480992}

\bibitem{Fulton_book}
W.~Fulton, \emph{Introduction to toric varieties}, Annals of Mathematics
  Studies, vol. 131, Princeton University Press, Princeton, NJ, 1993, The
  William H. Roever Lectures in Geometry. \MR{1234037}

\bibitem{Gross_SLag_I}
M.~Gross, \emph{Special {L}agrangian fibrations. {I}. {T}opology}, Integrable
  systems and algebraic geometry ({K}obe/{K}yoto, 1997), World Sci. Publ.,
  River Edge, NJ, 1998, pp.~156--193. \MR{1672120}

\bibitem{GHS16}
M.~Gross, P.~Hacking, and B.~Siebert, \emph{Theta functions on varieties with
  effective anti-canonical class}, Mem. Amer. Math. Soc., to appear,
  \href{https://arxiv.org/abs/1601.07081}{arXiv:1601.07081}.

\bibitem{GS_theta_function}
M.~Gross and B.~Siebert, \emph{Theta functions and mirror symmetry}, Surveys in
  differential geometry 2016. {A}dvances in geometry and mathematical physics,
  Surv. Differ. Geom., vol.~21, Int. Press, Somerville, MA, 2016, pp.~95--138.
  \MR{3525095}

\bibitem{Guillemin_book}
V.~Guillemin, \emph{Moment maps and combinatorial invariants of {H}amiltonian
  {$T^n$}-spaces}, Progress in Mathematics, vol. 122, Birkh\"auser Boston,
  Inc., Boston, MA, 1994. \MR{1301331}

\bibitem{Kanazawa_GQ}
A.~Kanazawa, \emph{Degenerations and {L}agrangian fibrations of {C}alabi-{Y}au
  manifolds}, Handbook for Mirror Symmetry of Calabi-Yau and Fano Manifolds,
  Adv. Lect. Math. (ALM), vol.~47, Int. Press, Somerville, MA, 2019,
  pp.~149--204.

\bibitem{HMS}
M.~Kontsevich, \emph{Homological algebra of mirror symmetry}, Proceedings of
  the {I}nternational {C}ongress of {M}athematicians, {V}ol.\ 1, 2 ({Z}\"urich,
  1994), Birkh\"auser, Basel, 1995, pp.~120--139. \MR{1403918}

\bibitem{KS_HMS_torus_fibration}
M.~Kontsevich and Y.~Soibelman, \emph{Homological mirror symmetry and torus
  fibrations}, Symplectic geometry and mirror symmetry ({S}eoul, 2000), World
  Sci. Publ., River Edge, NJ, 2001, pp.~203--263. \MR{1882331}

\bibitem{Leung05}
N.~C. Leung, \emph{Mirror symmetry without corrections}, Comm. Anal. Geom.
  \textbf{13} (2005), no.~2, 287--331. \MR{2154821}

\bibitem{LYZ}
N.~C. Leung, S.-T. Yau, and E.~Zaslow, \emph{From special {L}agrangian to
  {H}ermitian-{Y}ang-{M}ills via {F}ourier-{M}ukai transform}, Adv. Theor.
  Math. Phys. \textbf{4} (2000), no.~6, 1319--1341. \MR{1894858}

\bibitem{Ma_thesis}
Z.~N. Ma, \emph{From {W}itten-{M}orse {T}heory to {M}irror {S}ymmetry},
  ProQuest LLC, Ann Arbor, MI, 2014, Thesis (Ph.D.)--The Chinese University of
  Hong Kong (Hong Kong). \MR{3347168}

\bibitem{Markus_conjecture}
L.~Markus, \emph{Cosmological models in differential geometry}, Mimeographed
  Notes, Univ. of Minnesota (1962), p. 58.

\bibitem{Nohara}
Y.~Nohara, \emph{{L}agrangian fibrations and theta functions}, preprint (2006),
  \href{https://arxiv.org/abs/math/0604330}{arXiv:math/0604330}.

\bibitem{cohomology_GQ}
J.~\'Sniatycki, \emph{On cohomology groups appearing in geometric
  quantization},  (1977), 46--66. Lecture Notes in Math., Vol. 570.
  \MR{0451304}

\bibitem{SYZ}
A.~Strominger, S.-T. Yau, and E.~Zaslow, \emph{Mirror symmetry is
  {$T$}-duality}, Nuclear Phys. B \textbf{479} (1996), no.~1-2, 243--259.
  \MR{1429831}

\bibitem{Tyurin}
A.~Tyurin, \emph{Geometric quantization and mirror symmetry}, preprint (1999),
  \href{https://arxiv.org/abs/math/9902027}{arXiv:9902027}.

\bibitem{Witten_Morse}
E.~Witten, \emph{Supersymmetry and {M}orse theory}, J. Differential Geom.
  \textbf{17} (1982), no.~4, 661--692 (1983). \MR{683171}

\end{thebibliography}

\end{document}